\documentclass{amsart}
\usepackage{amssymb}
\usepackage{color}
\usepackage{graphicx}
\theoremstyle{plain}
\newtheorem{theorem}{Theorem}[section]
\newtheorem{corollary}[theorem]{Corollary}

\newtheorem{proposition}[theorem]{Proposition}
\newtheorem{conjecture}[theorem]{Conjecture}
\newtheorem{lemma}[theorem]{Lemma}
\newtheorem{rem}[theorem]{Remark}
\newtheorem{ques}[theorem]{Question}
\numberwithin{equation}{section}
\def\op{\operatorname}

\newcommand{\Z}{\mathbb{Z}}
\newcommand{\Q}{\mathbb{Q}}

\begin{document}

\title[Arithmetic properties of polynomials]{Arithmetic properties of
polynomial solutions of the Diophantine equation $P(x)x^{n+1}+Q(x)(x+1)^{n+1}=1$}

\author{Karl Dilcher}
\address{Department of Mathematics and Statistics\\
         Dalhousie University\\
         Halifax, Nova Scotia, B3H 4R2, Canada}
\email{dilcher@mathstat.dal.ca}

\author{Maciej Ulas}
\address{Institute of Mathematics of the Polish Academy of Sciences, \'{S}wi\c{e}tego Tomasza 30, 31-014 Krak\'{o}w }
\address{and}
\address{Jagiellonian University, Faculty of Mathematics and Computer Science,
Institute of Mathematics, \L{}ojasiewicza 6, 30-348 Krak\'ow, Poland}
\email{Maciej.Ulas@im.uj.edu.pl}

\thanks{Research supported in part by the Natural Sciences and Engineering
        Research Council of Canada, Grant \# 145628481}

\date{}

\setcounter{equation}{0}

\begin{abstract}
For each integer $n\geq 1$ we consider the unique polynomials $P, Q\in\Q[x]$
of smallest degree $n$ that are solutions of the equation
$P(x)x^{n+1}+Q(x)(x+1)^{n+1}=1$. We derive numerous properties of these
polynomials and their derivatives, including explicit expansions, differential
equations, recurrence relations, generating functions, resultants,
discriminants, and irreducibility results. We also consider some related
polynomials and their properties.
\end{abstract}

\maketitle

\section{Introduction}\label{sec1}

It is a well-known fact that the Chebyshev polynomials of the first and second
kind, $T_n(x)$ and $U_n(x)$, can be defined as solutions of the polynomial
Pell equation
\begin{equation}\label{1.1}
T_n(x)^2 - (x^2-1)U_{n-1}(x)^2 = 1
\end{equation}
in the ring $\Z[x]$; see \cite{DP}, or \cite{De} for more general
polynomial rings.

In this paper we consider the following variant of equation \eqref{1.1}. Since
$\Q[x]$ is a Euclidean domain, we know that for given coprime
polynomials $f,g\in\Q[x]$ there are polynomials $P, Q\in\Q[x]$
with $P(x)f(x)+Q(x)g(x)=1$. To make this more specific, we choose $f$ and $g$
to be the simplest pair of coprime polynomials of the same degree, namely
$x^{n+1}$ and $(x+1)^{n+1}$, where $n\geq 0$ is an integer. In other words, we
consider the equation
\begin{equation}\label{1.2}
P(x)x^{n+1}+Q(x)(x+1)^{n+1}=1.
\end{equation}
If we make the assumption that $\deg P\leq n$, $\deg Q\leq n$, then we have
a unique solution $P(x)=P_{n}(x), Q(x)=Q_{n}(x)$ of \eqref{1.2}, and we have
\begin{equation}\label{1.3}
P_{n}(x)x^{n+1}+Q_{n}(x)(x+1)^{n+1}=1
\end{equation}
for integers $n\geq 0$.

It is the purpose of this paper to study properties of the polynomial
sequences $P_n(x)$, $Q_n(x)$. We will see that these polynomials have integer
coefficients, are similar to each other, and in spite of some fundamental
differences they show some similarities with the Chebyshev polynomials in
equation \eqref{1.1}. Moreover, these polynomials appeared in an interesting
context of constructing consecutive integers divisible by high powers of their
largest prime factors \cite[Theorem 4]{KonMoi}.

The first polynomials $P_n(x)$, $Q_n(x)$, for $0\leq n\leq 4$, are shown in Table~1.

\bigskip
\begin{center}
\begin{tabular}{|c|r|r|}
\hline
$n$ & $P_n(x)$ & $Q_n(x)$ \\
\hline
0 & $-1$ & 1 \\
1 & $2x+3$ & $-2x+1$ \\
2 & $-6x^2-15x-10$ & $6x^2-3x+1$\\
3 & $20x^3+70x^2+84x+35$ & $-20x^3+10x^2-4x+1$ \\
4 & $-70x^4-315x^3-540x^2-420x-126$ & $70x^4-35x^3+15x^2-5x+1$  \\
\hline
\end{tabular}

\medskip
{\bf Table~1}: $P_n(x)$ and $Q_n(x)$ for $0\leq n\leq 4$.
\end{center}
\bigskip

This paper is structured as follows. We begin in Section~\ref{sec2} by deriving
some basic properties of the polynomials $P_n(x)$ and $Q_n(x)$, including
differential equations, recurrence relations, and generating functions.
In Sections~3 and~4 we consider variants of the original questions concerning
the identities \eqref{1.2} and \eqref{1.3}; this will involve the Chebyshev
polynomials already mentioned in connection with \eqref{1.1}. In Sections~5
and~6 we study resultants and discriminants involving the polynomials $Q_n(x)$
and their derivatives. We then introduce,
in Section~7, a sequence of polynomials related to the sequence $Q_n(x)$ and
study their properties. Section~8 is devoted to some irreducibility results,
and we conclude this paper with a few further remarks and conjectures in
Section~9.

\section{Basic Properties}\label{sec2}

Throughout the remainder of this paper, $P_n(x)$ and $Q_n(x)$ will denote the
solutions of the equation \eqref{1.3}.

\begin{proposition}\label{prop:2.1}
For any integer $n\geq 0$ we have $\deg P_n=\deg Q_n$, and
\begin{equation}\label{2.1}
P_{n}(x)=(-1)^{n+1}Q_{n}(-1-x), \qquad  Q_{n}(x)=(-1)^{n+1}P_{n}(-1-x).
\end{equation}
\end{proposition}

\begin{proof}
The definition \eqref{1.3} implies that the degrees of $P_n$ and $Q_n$ must be
the same. Now replace $x$ by $-1-x$ in \eqref{1.3}; then we get
\[
P_{n}(-1-x)(-1)^{n+1}(x+1)^{n+1}+Q_{n}(-1-x)(-1)^{n+1}x^{n+1}=1.
\]
Now the uniqueness of the solutions of \eqref{1.3} implies the two identities
in \eqref{2.1}.
\end{proof}

As a consequence of \eqref{1.3} and \eqref{2.1} we immediately get a few
special values.

\begin{corollary}\label{cor:2.2}
For any integer $n\geq 0$ we have
\begin{equation}\label{2.2}
P_{n}(-1)=(-1)^{n+1},\quad P_{n}(-\tfrac{1}{2})=(-1)^{n+1}2^n,\quad
Q_{n}(-\tfrac{1}{2})=2^n,\quad Q_{n}(0)=1.
\end{equation}
\end{corollary}

\begin{proof}
By substituting $x=-1$ in \eqref{1.3} we immediately get the first identity,
and similarly, $x=0$ gives the fourth identity. Next, if we set
$x=-\tfrac{1}{2}$ in \eqref{1.3}, we get
\[
(-1)^{n+1}P_{n}(-\tfrac{1}{2})+Q_{n}(-\tfrac{1}{2})=2^{n+1}.
\]
Also, either one of the equations in \eqref{2.1}, with $x=-\tfrac{1}{2}$,
gives $(-1)^{n+1}P_{n}(-\tfrac{1}{2})=Q_{n}(-\tfrac{1}{2})$. By combining
these last two identities we immediately get the third equation in \eqref{2.2},
and then also the second one.
\end{proof}

Another special value will be obtained later in this section. An important
consequence of Proposition~\ref{prop:2.1} is the fact that it suffices to
consider only one of $P_n$ and $Q_n$. For the remainder of this paper we will
therefore concentrate on the polynomial sequence $Q_n$, and begin by deriving
an explicit expression.

\begin{proposition}\label{prop:2.3}
For any integer $n\geq 0$ we have $\deg Q_n=n$, and
\begin{equation}\label{2.3}
Q_{n}(x)=\sum_{i=0}^n(-1)^i\binom{n+i}{i}x^i.
\end{equation}
\end{proposition}

Before proving this representation, we not that the polynomials $Q_n(x)$ have
an interesting combinatorial interpretation: The $i$th coefficient of
$Q_n(-x)$ counts the numbers of lattice paths from $(0,0)$ to $(n,i)$ using
only the steps $(1,0)$ and $(0,1)$; see \cite[A046899]{OEIS}.

Given the simple form of the representation \eqref{2.3}, it is not surprising
that the polynomials $Q_n(-x)$ have been considered before. Using \eqref{2.3}
as definition, Gould \cite{Go1} derived some basic properties, including the
third identity in \eqref{2.2} and two more properties relevant to this paper;
those will be mentioned later in this section.

\begin{proof}[Proof of Proposition~\ref{prop:2.3}]
We differentiate both sides of the identity \eqref{1.3}), and after some
easy manipulation we get
$$
x^n\left(xP'_{n}(x)+(n+1)P_{n}(x)\right)
=-(x+1)^n\left((x+1)Q'_{n}(x)+(n+1)Q_{n}(x)\right).
$$
Since $x^{n}$ and $(x+1)^{n}$ are coprime, this last identity implies that
$x^n$ divides the expression $(x+1)Q'_{n}(x)+(n+1)Q_{n}(x)$. Consequently,
since $\deg Q_n\leq n$, we have
\begin{equation}\label{2.4}
(x+1)Q'_{n}(x)+(n+1)Q_{n}(x) = c_nx^n
\end{equation}
for some constant $c_n$. Now we write
\begin{equation}\label{2.5}
Q_{n}(x)=\sum_{i=0}^n q_{i,n}x^i
\end{equation}
and equate coefficients of $x^i$ on both sides of \eqref{2.4}. First we have
by Corollary~\ref{cor:2.2} that $q_{0,n}=Q_n(0)=1$, and for $i=0,1,\ldots,n-1$
we have
$$
iq_{i,n}+(i+1)q_{i+1,n}+(n+1)q_{i,n}=0,
$$
or equivalently
$$
q_{i+1,n}=-\frac{n+i+1}{i+1}q_{i,n},\quad i=0,1,\ldots,n-1.
$$
By iterating this, we immediately get, for $i=0,1,\ldots,n$,
$$
q_{i,n}=(-1)^{i}\frac{(n+1)\cdots (n+i)}{i!}=(-1)^{i}\binom{n+i}{i}.
$$
This, with \eqref{2.5}, completes the proof.
\end{proof}

By combining the identities \eqref{2.3} and \eqref{2.1}, we can also obtain
an explicit expression for $P_n(x)$.

\begin{corollary}\label{cor:2.4}
For any integer $n\geq 0$ we have $\deg P_n=n$, and
\begin{equation}\label{2.6}
P_{n}(x)=(-1)^{n+1}(2n+1)\binom{2n}{n}\sum_{i=0}^{n}\frac{1}{n+i+1}\binom{n}{i}x^{i}.
\end{equation}
\end{corollary}

\begin{proof}
Combining the first identity in \eqref{2.1} with \eqref{2.3}, we get
\begin{align*}
(-1)^{n+1}P_{n}(x)&=Q_{n}(-1-x)=\sum_{i=0}^{n}(-1)^{i}\binom{n+i}{i}(-1-x)^{i}\\
&=\sum_{i=0}^{n}\binom{n+i}{i}(x+1)^{i}
=\sum_{i=0}^{n}\binom{n+i}{i}\sum_{j=0}^{i}\binom{i}{j}x^{j}\\
&=\sum_{j=0}^{n}\left(\sum_{i=0}^{n}\binom{n+i}{i}\binom{i}{j}\right)x^{j},
\end{align*}
where we have used the fact that $\binom{i}{j}=0$ for $j>i$, and then changed
the order of summation. Finally we note that the inner sum in the last
expansion is known to have the evaluation
$\frac{2n+1}{n+j+1}\binom{2n}{n}\binom{n}{j}$; for instance, after an easy
reformulation of the sum one could use identity (1.48) in \cite{Go}. This
proves \eqref{2.6}.
\end{proof}

\begin{rem}\label{rem:2.5}
{\rm The identity \eqref{2.6} immediately gives the special values}
\begin{equation}\label{2.6a}
P_n(0) = (-1)^{n+1}\frac{2n+1}{n+1}\binom{2n}{n}
= (-1)^{n+1}\binom{2n+1}{n+1} = (-1)^{n+1}Q_n(-1).
\end{equation}
{\rm On the other hand, there do not seem to exist explicit formulas for $P_n(1)$
and $Q_n(1)$. However, the two sequences of their absolute values, namely
$(1, 5, 31, 209, 1471,$ $10625, \ldots)$ and $(1, 1, 4, 13, 46, 166, \ldots)$,
have some interesting combinatorial interpretations; see the entries
A178792 and A026641, respectively, in \cite{OEIS}. Also, it follows from
\eqref{1.3} that} $P_n(1)+2^{n+1}Q_n(1)=1$.
\end{rem}

We now return to the identity \eqref{2.4} and note that
$c_n=(2n+1)q_{n,n}=(-1)^n(2n+1)\binom{2n}{n}$, so that
\begin{equation}\label{2.7}
(x+1)Q'_{n}(x)+(n+1)Q_{n}(x) = (-1)^n(2n+1)\binom{2n}{n}x^n.
\end{equation}
The following generalization of this identity will be useful later. As usual,
$Q_n^{(k)}(x)$ will denote the $k$th derivative of $Q_n(x)$.

\begin{proposition}\label{prop:2.6}
For integers $k$ and $n$ with $1\leq k\leq n+1$ we have
\begin{equation}\label{2.8}
(x+1)Q_{n}^{(k)}(x)+(n+k)Q_{n}^{(k-1)}(x)=(-1)^n\frac{(2n+1)!}{n!}
\frac{x^{n-k+1}}{(n-k+1)!}.
\end{equation}
\end{proposition}

\begin{proof}
We proceed by induction on $k$. For $k=1$, the identity \eqref{2.8} reduces
to \eqref{2.7}. Now we suppose that \eqref{2.8} holds for some $k\geq 1$, and
differentiate both sides with respect to $x$. Simplifying the resulting
identity, we get
\begin{equation}\label{2.8a}
(x+1)Q_{n}^{(k+1)}(x)+(n+k+1)Q_{n}^{(k)}(x)
= (-1)^n\frac{(2n+1)!}{n!}\frac{x^{n-k}}{(n-k)!},
\end{equation}
which is the same as \eqref{2.8} with $k$ replaced by $k+1$. This completes
the proof by induction.
\end{proof}

As a first application of Proposition~\ref{prop:2.6} we set $x=-1$ in
\eqref{2.8a}, which gives the following special values.

\begin{corollary}\label{cor:2.7}
For integers $0\leq k\leq n$ we have
\begin{equation}\label{2.8b}
Q_{n}^{(k)}(-1)
= \frac{(-1)^k}{n+k+1}\cdot\frac{(2n+1)!}{n!(n-k)!}.
\end{equation}
\end{corollary}

We note that for $k=0$ this last identity is consistent with the right-hand
side of \eqref{2.6a}.
As another application of Proposition~\ref{prop:2.6} we obtain a homogeneous
differential equation satisfied by the polynomials $Q_n(x)$.

\begin{proposition}\label{prop:2.8}
For $n\geq 0$ we have
\begin{equation}\label{2.9}
x(x+1)Q_{n}''(x)+(2x-n)Q_{n}'(x)-n(n+1)Q_{n}(x)=0.
\end{equation}
\end{proposition}

\begin{proof}
We use \eqref{2.8} with $k=2$ and multiply both sides by $x$, obtaining
\begin{equation}\label{2.10}
x(x+1)Q''_{n}(x)+(n+2)xQ'_{n}(x)=(-1)^n\frac{(2n+1)!}{n!}\frac{x^n}{(n-1)!}.
\end{equation}
Then we use \eqref{2.8} again, this time with $k=1$, and multiply both sides
by $n$, which gives
\[
n(x+1)Q'_{n}(x)+n(n+1)Q_{n}(x))=(-1)^n\frac{(2n+1)!}{n!}\frac{x^n}{(n-1)!}.
\]
By subtracting this last identity from \eqref{2.10}, we immediately get
\eqref{2.9}.
\end{proof}

Next we present a second-order linear recurrence relation satisfied by the sequence of polynomials $Q_{n}^{(k)}(x)$.

\begin{proposition}\label{prop:2.9}
Let $k\geq 0$ be an integer, and set
\begin{align*}
u_{k,n}(x)&=n(n+k)(x+1)\big(2(n+k-1)x+n+2k-1\big),\\
v_{k,n}(x)&=-2(n+k-1)(2(n+k)-1)x^2\big(2(n+k)x+3n+4k\big)\\
          &\qquad-(n+2k-1)(n+2k)\big(2(n+k-1)x-n\big),\\
w_{k,n}(x)&=2(n+2k-1)(2(n+k)-1)x\big(2(n+k)x+n+2k\big).
\end{align*}
Then the polynomials $Q_{k+n}^{(k)}(x)$ satisfy the recurrence relation
\begin{equation}\label{2.11}
u_{k,n}(x)Q_{k+n}^{(k)}(x)=v_{k,n}(x)Q_{k+n-1}^{(k)}(x)
+w_{k,n}(x)Q_{k+n-2}^{(k)}(x)\qquad (n\geq 2),
\end{equation}
with initial conditions
\[
Q_{k}^{(k)}(x)=(-1)^{k}\frac{(2k)!}{k!}\quad\hbox{and}\quad
Q_{k+1}^{(k)}(x)=(-1)^{k}\frac{(2k+1)!}{(k+1)!}(1-2(k+1)x).
\]
\end{proposition}

In the special case $k=0$ the expressions $u_{0,n}, v_{0,n}, w_{0,n}$ have the
common factor $(2x+1)n(n-1)$, and we obtain the following much simpler
recurrence relation.

\begin{corollary}\label{cor:2.10}
The polynomials $Q_{n}(x)$ satisfy
\begin{equation}\label{2.11a}
n(x+1)Q_{n}(x)=-(2(2n-1)x^2+2(2n-1)x-n)Q_{n-1}(x)+2(2n-1)xQ_{n-2}(x)
\end{equation}
for $n\geq 2$, with initial conditions $Q_{0}(x)=1$ and $Q_{1}(x)=-2x+1$.
\end{corollary}

\begin{proof}[Proof of Proposition~\ref{prop:2.9}]
By differentiating the explicit formula \eqref{2.3}, we obtain
\begin{equation}\label{2.11b}
Q_{k+n}^{(k)}(x)
=\sum_{i=0}^{n}(-1)^{i+k}\frac{(i+k)!}{i!}\binom{n+i+2k}{i+k}x^{i},
\end{equation}
which immediately gives the initial conditions. A straightforward but tedious
computation, using \eqref{2.11b}, then shows that the recurrence relation
\eqref{2.11} is satisfied. We leave the details to the reader.
\end{proof}

\begin{rem}\label{rem:2.11}
{\rm While the above proof would be sufficient, a few words about the discovery
of the recurrence relations are perhaps in order.
The relation \eqref{2.11a} was first guessed with the Maple package {\tt EKHAD}
written by D. Zeilberger, which can be obtained through the online supplement
to the book \cite{PWZ}. In particular, the procedure {\tt findrec}, applied to
the explicit formula \eqref{2.3}, gives the second degree difference operator
\begin{align*}
&((1+n)N+2x(1+2n))(N(x+1)-1)\\
&\qquad=(1 + n) (1 + x)N^{2}+(2(2n+1)x^2+2(2n+1)x-n-1)N-2(1 + 2 n)x,
\end{align*}
where $N$ is the forward unit shift operator, i.e., $N(f_{n})=f_{n+1}$.
Replacing $N^{i}$ by $Q_{n+i}(x)$ and then shifting $n$ to $n-2$ gives the
recurrence relation \eqref{2.11a}. The correctness of Zeilberger's algorithm
gives an alternative proof that the recurrence is satisfied for all $n$.

Given this nice result, we suspected that similar relations should hold also
for the $k$th derivative of the polynomials $Q_{n+k}(x)$. We checked that for
each fixed small $k$ we have a formula similar to \eqref{2.11a}, and then
guessed the general form \eqref{2.11} from the particular cases.
}
\end{rem}

At this point it should also be mentioned that Gould \cite{Go1} derived a
first-order inhomogeneous recurrence relation for $Q_n(-x)$; in our notation
it can be written as
\[
(x+1)Q_{n+1}(x) = Q_n(x) + \binom{2n+1}{n+1}(2x+1)(-x)^{n+1}.
\]
This identity was then used in \cite{Go1} to obtain the following generating
function. For the sake of completeness we give a different proof, based on
Corollary~\ref{cor:2.10}.

\begin{proposition}\label{prop:2.10}
The ordinary generating function for the polynomials $Q_n(x)$ is given by
\begin{equation}\label{2.12}
\frac{1+4xt+(1+2x)\sqrt{1+4xt}}{2(1+x-t)(1+4xt)}=\sum_{n=0}^{\infty}Q_{n}(x)t^n.
\end{equation}
\end{proposition}

\begin{proof}
We denote the expression on the right of \eqref{2.12} by $\mathcal{Q}(x,t)$, and
let $\mathcal{Q}'(x,t)$ be the derivative with respect to $t$. Then we easily
obtain the identities
\begin{align*}
\sum_{n=0}^{\infty}nQ_{n}(x)t^{n}  &=t\mathcal{Q}'(x,t),\\
\sum_{n=0}^{\infty}nQ_{n+1}(x)t^{n}&=\frac{1}{t}(t\mathcal{Q}'(x,t)-\mathcal{Q}(x,t)+1),\\
\sum_{n=0}^{\infty}nQ_{n+2}(x)t^{n}&=\frac{1}{t^2}(t\mathcal{Q}'(x,t)-2\mathcal{Q}(x,t)+2(1-2x)t).
\end{align*}
We now use these identities and apply standard methods to transform the
recurrence relation \eqref{2.11}, with $n$ replaced by $n+2$, into the
differential equation
\begin{align*}
&(t-x-1)(4tx+1)(4tx-t-1)\mathcal{Q}'(x,t)\\
&\qquad+(6xt-2x^2-2x+1)((4x-1)t-1)\mathcal{Q}(x,t)+x((4x-1)t-1)=0.
\end{align*}
Using standard methods for solving first-degree linear differential equations,
we get the general solution
$$
\mathcal{Q}(x,t)=\frac{1}{2 (1+x-t)}+\frac{c_1}{(1+x-t) \sqrt{1+4xt}}.
$$
The initial condition $\mathcal{Q}'(x,0)=1$ leads to $c_{1}=\frac{1}{2}(2x+1)$, and
we get the desired solution
$$
\mathcal{Q}(x,t)=\frac{1}{2(1+x-t)}\left(\frac{\sqrt{1+4xt}+2x+1}{\sqrt{1+4xt}}\right).
$$
Finally, we obtain \eqref{2.12} by multiplying numerator and denominator of this
last equation by $\sqrt{1+4xt}$.
\end{proof}

\section{A Pell-type polynomial equation}\label{sec3}

In this section and the next we will, more generally, consider the equation
\begin{equation}\label{3.1}
P_n(x)Y(x)^{n+1} + Q_n(x)Z(x)^{n+1} = 1
\end{equation}
for integers $n\geq 0$ and unknown polynomials $Y, Z\in{\mathbb Q}[x]$. We
began this paper from the point of view of fixed polynomials $Y(x)=x$ and
$Z(x)=x+1$; with the conditions $\deg P_n\leq n$, $\deg Q_n\leq n$ this had led
us to the unique polynomial sequences $P_n, Q_n$ in Sections~1 and~2.

We now ask the following ``inverse question": Given our polynomials $P_n, Q_n$,
what can we say about polynomial solutions $Y, Z$ of \eqref{3.1}? Are there
solutions other than $Y(x)=x, Z(x)=x+1$?

The case $n=0$ is of no interest since $P_0=-1, Q_0=1$ implies that for any
$Y\in{\mathbb Q}[x]$ (or ${\mathbb Z}[x]$) there is a $Z\in{\mathbb Q}[x]$
(or ${\mathbb Z}[x]$) satisfying \eqref{3.1}. Hence we assume $n\geq 1$.

It turns out that the case $n=1$ is of particular interest, and is very
different from the case $n\geq 2$. This section will be devoted to this special
case; that is, since $P_1(x)=2x+3$ and $Q_1(x)=1-2x$, we consider the
equation
\begin{equation}\label{3.2}
(2x+3)Y(x)^2 + (1-2x)Z(x)^2 = 1,
\end{equation}
in unknown polynomials $Y, Z\in{\mathbb Q}[x]$.

\begin{proposition}\label{prop:3.1}
The equation \eqref{3.2} has infinitely many solutions which, for $n\geq 0$,
are given by
\begin{equation}\label{3.3}
\begin{cases}
&Y(x)=Y_n(x)
= \frac{1}{2}\left(U_{n+1}(x+\tfrac{1}{2})-U_n(x+\tfrac{1}{2})\right),\\
&Z(x)=Z_n(x)
= \frac{1}{2}\left(U_{n+1}(x+\tfrac{1}{2})+U_n(x+\tfrac{1}{2})\right),
\end{cases}
\end{equation}
where $U_n(y)$ are the Chebyshev polynomials of the second kind. Furthermore,
for $n\geq 1$ we have the recurrence relations
\begin{equation}\label{3.4}
\begin{cases}
&Y_{n+1}(x) = (2x+1)Y_n(x) - Y_{n-1}(x), \\
&Z_{n+1}(x) = (2x+1)Z_n(x) - Z_{n-1}(x),
\end{cases}
\end{equation}
with initial conditions $Y_0(x)=x$, $Y_1(x)=2x^2+x-\tfrac{1}{2}$,
$Z_0(x)=x+1$, and $Z_1(x)=2x^2+3x+\tfrac{1}{2}$.
\end{proposition}

\begin{proof}
We transform the equation \eqref{3.2} to a Pell-type equation by multiplying
both sides by $2x+3$. Noting that
\[
(2x-1)(2x+3) = 4\left((x+\tfrac{1}{2})^2-1\right),
\]
we then get the equation
\begin{equation}\label{3.5}
\bigl((2x+3)Y(x)\bigr)^2-\left((x+\tfrac{1}{2})^2-1\right)(2Z(x))^2 = 2x+3.
\end{equation}
We now use the following known results from the theory of Pell equations; see,
e.g., \cite[p.~354]{NZM}.

Suppose that $(x,y)=(r_0,s_0)$ is a solution of the Pell-type equation
\begin{equation}\label{3.6}
x^2 - dy^2 = N,
\end{equation}
where $d, N\in{\mathbb Z}$, and $d>1$ is not a perfect square. Furthermore,
suppose that $(x,y)=(x_n,y_n)$, $n-1, 2, 3,\ldots$, are the solutions of the
Pell equation
\begin{equation}\label{3.7}
x^2 - dy^2 = 1.
\end{equation}
Then the pairs $(r_n,s_n)$, defined by
\begin{equation}\label{3.8}
r_n+s_n\sqrt{d} = (r_0+s_0\sqrt{d})(x_n+y_n\sqrt{d}),
\end{equation}
are also solutions of \eqref{3.6}. We note that the equation \eqref{3.6} may or
may not have solutions, while \eqref{3.7} always has infinitely many solutions
that can all be given, for instance by way of recurrence relations. By
expanding the right-hand side of \eqref{3.8} and equating rational and
irrational terms, we immediately get
\begin{equation}\label{3.9}
r_n = r_0x_n + d\,s_0y_n,\qquad s_n = s_0x_n + r_0y_n.
\end{equation}
Comparing \eqref{3.5} with \eqref{3.6} and noting that $Y(x)=x$ and $Z(x)=x+1$
is a solution of \eqref{3.2}, we have
\begin{equation}\label{3.10}
r_0=x(2x+3),\quad s_0=2(x+2),\quad d=(x+\tfrac{1}{2})^2-1,\quad N=2x+3.
\end{equation}
Furthermore, comparing \eqref{3.5} (having 1 instead of $2x+3$ on the right)
with \eqref{1.1}, we see that
\begin{equation}\label{3.11}
x_n = T_n(x+\tfrac{1}{2}),\qquad y_n = U_{n-1}(x+\tfrac{1}{2}),\quad n=1,2,\ldots
\end{equation}
This holds also for $n=0$ since $T_0(y)=1$ and, by convention, $U_{-1}(y)=0$.
Finally, with $r_n=(2x+3)Y(x)$ and $s_n=2Z(x)$, the identities
\eqref{3.9}, \eqref{3.10}, \eqref{3.11} yield
\begin{equation}\label{3.12}
\begin{cases}
&Y(x)=Y_n(x) = xT_n(x+\tfrac{1}{2})
+\tfrac{1}{2}(2x-1)(x+1)U_{n-1}(x+\tfrac{1}{2}),\\
&Z(x)=Z_n(x) = (x+1)T_n(x+\tfrac{1}{2})
+\tfrac{1}{2}x(2x+3)U_{n-1}(x+\tfrac{1}{2}).
\end{cases}
\end{equation}
The initial conditions following
\eqref{3.4} can be computed from \eqref{3.12}, using $T_0(y)=U_0(y)=1$ and
$T_1(y)=y$. Furthermore, the well-known recurrence relation for the Chebyshev
polynomials, namely $T_{n+1}(x)=2xT_n(x)-T_{n-1}(x)$ leads to
$$
T_{n+1}(x+\tfrac{1}{2})=(2x+1)T_n(x+\tfrac{1}{2})-T_{n-1}(x+\tfrac{1}{2}),
$$
with the same relation also for the sequence $U_n(x+\tfrac{1}{2})$. This means
that the polynomial sequences $Y_n(x)$ and $Z_n(x)$ in \eqref{3.12} satisfy
this relation as well.

Finally, using the same argument as in the previous paragraph, we see that the
right-hand sides of the expressions in \eqref{3.3} satisfy \eqref{3.4} as well.
Using the fact that $U_1(y)=2y$ and $U_2(y)=4y^2-1$, we can easily verify that
the terms on the right in \eqref{3.3} also satisfy the initial conditions
following \eqref{3.4}. This proves the identities in \eqref{3.3}, and the proof
of the the proposition is complete.
\end{proof}


\bigskip
\begin{center}
\begin{tabular}{|c|r|r|}
\hline
$n$ & $Y_n(x)$ & $Z_n(x)$ \\
\hline
0 & $x$ & $x+1$ \\
1 & $2x^2+x-\tfrac{1}{2}$ & $2x^2+3x+\tfrac{1}{2}$ \\
2 & $4x^3+4x^2-x-\tfrac{1}{2}$ & $4x^3+8x^2+3x-\tfrac{1}{2}$\\
3 & $x(8x^3+12x^2-3)$ & $(x+1)(8x^3+12x^2-1)$ \\
4 & $16x^5+32x^4+8x^3-10x^2-2x+\tfrac{1}{2}$ & $16x^5+48x^4+40x^3+2x^2+6x-\tfrac{1}{2}$  \\
\hline
\end{tabular}

\medskip
{\bf Table~2}: $Y_n(x)$ and $Z_n(x)$ for $0\leq n\leq 4$.
\end{center}
\bigskip

We can now use Proposition~3.1 to show that the polynomials $Y_n$ and $Z_n$ are
``almost in ${\mathbb Z}[x]$", in the following sense.

\begin{corollary}\label{cor:3.2}
Let $n\geq 0$ be an integer. Then

$(a)$ $\;Y_n(x)\in{\mathbb Z}[x]$ and $Z_n(x)\in{\mathbb Z}[x]$ when
$n\equiv 0\pmod{3}$;

$(b)$ $\;Y_n(x)+\tfrac{1}{2}\in{\mathbb Z}[x]$ and
$Z_n(x)+\tfrac{1}{2}\in{\mathbb Z}[x]$ when $n\not\equiv 0\pmod{3}$.

$(c)$ More specifically, we have
\[
Y_n(0)=\begin{cases}
-\tfrac{1}{2} &\hbox{if}\; n\equiv 1, 2\pmod{6},\\
\tfrac{1}{2} &\hbox{if}\; n\equiv 4, 5\pmod{6},\\
0 &\hbox{if}\; n\equiv 0\pmod{3};
\end{cases}
\quad
Z_n(0)=\begin{cases}
-\tfrac{1}{2} &\hbox{if}\; n\equiv 2, 4\pmod{6},\\
\tfrac{1}{2} &\hbox{if}\; n\equiv 1, 5\pmod{6},\\
(-1)^k &\hbox{if}\; n=3k.
\end{cases}
\]
\end{corollary}

\begin{proof}
We first note that $Y_n$ and $Z_n$ satisfy all three statements for $n=0$ and
$n=1$. Next, the coefficient $2x+1$ in the recurrence relations \eqref{3.4}
guarantees that by induction on $n$ we have that all coefficients of all
$Y_n, Z_n$, with the possible exception of the constant coefficients, are
integers.

Finally, it is clear from \eqref{3.4} that the constant coefficients of $Y_n$
satisfy the recurrence relation $Y_{n+1}(0)=Y_n(0)-Y_{n-1}(0)$, and similarly
for the sequence $Z_n$. Part (c) then follows from a simple induction, and
the remaining statements in (a) and (b) also follow.
\end{proof}

We observe in Table~2 that both $Y_3(x)$ and $Z_3(x)$ are reducible. Further
computations show that this seems to be the case for all $Y_{3n}(x)$ and
$Z_{3n}(x)$, a fact confirmed by the following result.

\begin{proposition}\label{prop:3.3}
For all integers $n\geq 1$, the polynomials $Y_{3n}(x)$ and $Z_{3n}(x)$ are
reducible. Specifically, if we set $y:=x+\tfrac{1}{2}$ for simplicity, we have
\begin{equation}\label{3.13}
\begin{cases}
&Y_{3n}(x)=\frac{1}{2}\left(U_n(y)-U_{n-1}(y)\right)
\left(U_{2n+1}(y)-U_{2n-1}(y)-1\right)\\
&Z_{3n}(x)=\frac{1}{2}\left(U_n(y)+U_{n-1}(y)\right)
\left(U_{2n+1}(y)-U_{2n-1}(y)+1\right)\\
\end{cases}
\end{equation}
\end{proposition}

\begin{proof}
In view of \eqref{3.3}, the first identity in \eqref{3.13} is equivalent to
\begin{equation}\label{3.14}
U_{3n+1}(y)-U_{3n}(y)
= \left(U_n(y)-U_{n-1}(y)\right)\left(U_{2n+1}(y)-U_{2n-1}(y)-1\right).
\end{equation}
We now use the well-known identity
\begin{equation}\label{3.15}
2T_j(y)U_k(y) = U_{j+k}(y) - U_{j-k-2}(y),\qquad k\leq j-2
\end{equation}
(see, e.g., \cite{Ri}), and set $j=2n+1$, $k=0$ and note that $U_0(y)=1$.
With this, the right-hand side of \eqref{3.14} becomes
\begin{align*}
&\left(U_n(y)-U_{n-1}(y)\right)\left(2T_{2n+1}(y)-1\right) \\
&\qquad =2T_{2n+1}(y)U_n(y)-2T_{2n+1}(y)U_{n-1}(y)-U_n(y)+U_{n-1}(y) \\
&\qquad =\left(U_{3n+1}(y)-U_{n-1}(y)\right)-\left(U_{3n}(y)-U_n(y)\right)
-U_n(y)+U_{n-1}(y) \\
&\qquad = U_{3n+1}(y)-U_{3n}(y),
\end{align*}
where in the second equation above we have used \eqref{3.15} two more times.
This proves \eqref{3.14}

For the second part of \eqref{3.13}, we replace $y$ by $-y$ in \eqref{3.14} and
use the symmetry property $U_n(-y)=(-1)^nU_n(y)$. Then we get
\begin{align*}
&(-1)^{n+1}U_{3n+1}(y)-(-1)^nU_{3n}(y) \\
&\qquad= \left((-1)^nU_n(y)-(-1)^{n-1}U_{n-1}(y)\right)
\left(-U_{2n+1}(y)+U_{2n-1}(y)-1\right).
\end{align*}
Multiplying both sides by $(-1)^{n+1}$, we get
\[
U_{3n+1}(y)+U_{3n}(y)
= \left(U_n(y)+U_{n-1}(y)\right)\left(U_{2n+1}(y)-U_{2n-1}(y)+1\right),
\]
which is equivalent to the desired identity.
\end{proof}

By iterating the factorizations on the right of \eqref{3.13}, we obtain the
following easy consequence.

\begin{corollary}\label{cor:3.3}
For all integers $n\geq 1$ and $k\geq 1$, the polynomials $Y_{3^kn}(x)$ and
$Z_{3^kn}(x)$ have at least $k+1$ nonconstant factors.
\end{corollary}

\section{The equation \eqref{3.1} in general}\label{sec4}

In this section we return to the question posed at the beginning of
Section~\ref{sec3}. The following result deals with the general case. As before,
$P_n$ and $Q_n$ denote the polynomials studied in Sections~\ref{sec1}
and~\ref{sec2}, and given by the explicit formulas \eqref{2.3} and \eqref{2.6}.

\begin{proposition}\label{prop:4.1}
Let $n\geq 1$ be an integer, and consider the Diophantine equation
\begin{equation}\label{4.1}
P_n(x)Y(x)^{n+1} + Q_n(x)Z(x)^{n+1} = 1
\end{equation}
in unknown polynomials $Y, Z$.

\noindent
$(a)$ If $n=1$, then the equation \eqref{4.1} has infinitely many solutions
$Y, Z\in{\mathbb Q}[x]$.

\noindent
$(b)$ If $n\geq 2$, then the only solutions $Y, Z\in{\mathbb C}[x]$ of
\eqref{4.1} are $Y(x)=\zeta x$, $Z(x)=\xi(x+1)$, where $\zeta$ and $\xi$ are
arbitrary $(n+1)$th roots of unity.
\end{proposition}

\begin{proof}
Differentiating both sides of \eqref{4.1}, we get
\begin{align*}
&\left(P'_n(x)Y(x)+(n+1)P_n(x)Y'(x)\right)Y(x)^n\\
&\qquad\qquad=-\left(Q'_n(x)Z(x)+(n+1)Q_n(x)Z'(x)\right)Z(x)^n.
\end{align*}
The polynomials $Y$ and $Z$ must be coprime, so this implies
$$
\left.Y(x)^n\right| Q'_n(x)Z(x)+(n+1)Q_n(x)Z'(x),\quad
\left.Z(x)^n\right| P'_n(x)Y(x)+(n+1)P_n(x)Y'(x).
$$
Since $\deg{P_n}=\deg{Q_n}=n$, we clearly have
\[
\deg\left(P'_n(x)Y(x)+(n+1)P_n(x)Y'(x)\right)\leq \deg\left(P'_n(x)Y(x)\right)
= n-1+\deg{Y},
\]
and similarly with $P_n$ and $Y$ replaced by $Q_n$ and $Z$, respectively. Hence
\begin{equation}\label{4.2}
n-1+\deg{Y} \geq n\deg{Z}\quad\hbox{and}\quad n-1+\deg{Z} \geq n\deg{Y},
\end{equation}
and consequently $n-1+\deg{Y} \geq n(\deg{Y}-n+1)$, which is equivalent to
\[
(n^2-1) \geq (n^2-1)\deg{Y}.
\]
This last inequality holds only in the following cases:

(1) $\;n=1$.

(2) $\;n\geq 2$ and $\deg{Y}=0$; then by \eqref{4.2} also $\deg{Z}=0$.

(3) $\;n\geq 2$ and $\deg{Y}=1$; then by \eqref{4.2} also $\deg{Z}=1$.

The case (1) was covered by Proposition~\ref{prop:3.1}, and leads to part (a)
of this proposition. We therefore continue with $\deg{Y}=\deg{Z}=0$ and set
$Y=a$ and $Z=b$ with $a,b\in\mathbb C$. Then \eqref{4.1} becomes
\begin{equation}\label{4.3}
P_n(x)a^{n+1} + Q_n(x)b^{n+1} = 1\qquad(n\geq 2).
\end{equation}
We write the expressions \eqref{2.6} and \eqref{2.3} more explicitly as
\begin{equation}\label{4.4}
\begin{cases}
& P_n(x)=(-1)^{n+1}\binom{2n}{n}\left(x^n+(n+\tfrac{1}{2})x^{n-1}+\cdots\right),\\
& Q_n(x)=(-1)^{n+1}\binom{2n}{n}\left(-x^n+\tfrac{1}{2}x^{n-1}-\cdots\right).
\end{cases}
\end{equation}
Using this and equating coefficients of $x^{n+1}$ and of $x^n$ in \eqref{4.3}
we get, respectively,
\[
a^{n+1}-b^{n+1}=0\qquad\hbox{and}\qquad (2n+1)a^{n+1}+b^{n+1}=0.
\]
This system of equations has the unique solution $a=b=0$, which contradicts
\eqref{4.3}. Hence there are no solutions in case (2).

To deal with case (3), we set $Y(x)=ax+c$ and $Z(x)=bx+d$. Then \eqref{4.1}
becomes
\begin{align}
&P_n(x)\left(a^{n+1}x^{n+1}+(n+1)a^ncx^n+\cdots\right) \label{4.5} \\
&\qquad\qquad+ Q_n(x)\left(b^{n+1}x^{n+1}+(n+1)b^ndx^n+\cdots\right) = 1,\nonumber
\end{align}
where $n\geq 2$. First we equate coefficients of $x^{2n+1}$ in \eqref{4.5},
using \eqref{4.4}. This gives
\begin{equation}\label{4.6}
a^{n+1} = b^{n+1}\qquad (a\neq 0).
\end{equation}
Similarly, equating coefficients of $x^{2n}$ in \eqref{4.5}, we get
\[
(n+\tfrac{1}{2})a^{n+1}+(n+1)a^nc+\tfrac{1}{2}b^{n+1}-(n+1)b^nd = 0.
\]
We now multiply both sides of this last identity by $b$, use \eqref{4.6}, and
divide everything by $(n+1)a^n$. This gives
\begin{equation}\label{4.7}
a(b-d)+bc =0.
\end{equation}
Next we show that $c=0$. To do so, we set $x=0$ in \eqref{4.1} and use the
value of $P_(0)$ given in Remark~2.5, along with the fact that $Q_n(0)=1$
(see Corollary~\ref{cor:2.2}). Then we have
\begin{equation}\label{4.8}
(-1)^{n+1}\binom{2n+1}{n+1}c^{n+1} + d^{n+1} = 1\qquad (n\geq 2).
\end{equation}
First, with $n=2$ nd $n=5$, this gives
\[
-\binom{5}{3}c^3+d^3=1\quad\hbox{and}\quad \binom{11}{6}c^6+d^6=1.
\]
Substituting $d^3=1+10c^3$ into the second identity above, we obtain
\[
\left(\binom{11}{6}+100\right)c^6+20c^3 = 0,
\]
which means that either $c=0$, or
\begin{equation}\label{4.9}
\left(\binom{11}{6}+100\right)c^3 = -20.
\end{equation}
Second, with $n=3$ and $n=7$ in \eqref{4.8}, we find in the same way that
\[
\left(\binom{15}{8}+\binom{7}{4}^2\right)c^8-2\binom{7}{4}c^4 = 0,
\]
which again means that either $c=0$, or $rc^4=s$ for certain integers $r,s>0$.
But the arguments of the four possible solutions in $c\in\mathbb C$ are
different from those in \eqref{4.9}.
Therefore we have $c=0$, and consequently $b=d$ by \eqref{4.7}. This, with
\eqref{4.6}, finally gives part (b) of the result.
\end{proof}

To conclude this section, we consider a question related to the original
problem concerning the equation \eqref{1.2}. If instead of the polynomials
$x$ and $x+1$ we consider an arbitrary fixed pair of coprime polynomials
$Y,Z\in{\mathbb Z}[x]$, what can we say about possible solutions of the
equation
\begin{equation}\label{4.10}
p_n(x)Y(x)^{n+1} + q_n(x)Z(x)^{n+1} = 1\qquad (n\geq 0),
\end{equation}
in unknown polynomials $p_n, q_n\in{\mathbb Z}[x]$? We know that the existence
of solutions is not guaranteed since ${\mathbb Z}[x]$ is not a Euclidean
domain. However, we have the following result.

\begin{proposition}\label{prop:4.2}
Let $Y,Z\in{\mathbb Z}[x]$ be coprime polynomials and suppose that there exist
$p_0, q_0\in{\mathbb Z}[x]$ such that
\begin{equation}\label{4.11}
p_0(x)Y(x) + q_0(x)Z(x) = 1.
\end{equation}
Then for each integer $n\geq 1$ there are polynomials
$p_n, q_n\in{\mathbb Z}[x]$ such that \eqref{4.10} holds.
\end{proposition}

\begin{proof}
We proceed by induction on $n$. The induction beginning is given by the
hypothesis \eqref{4.11}, which corresponds to $n=0$. We now assume that for
some $n\geq 0$ there are polynomials $p_n, q_n\in{\mathbb Z}[x]$ such that
\eqref{4.10} holds, and multiply the equation with \eqref{4.11}. Then we get
\begin{equation}\label{4.12}
p_0p_nY^{n+2}+YZ\left(p_0q_nZ^n+q_0p_nY^n\right)+q_0q_nZ^{n+2} = 1,
\end{equation}
where for greater ease of notation we have suppressed the variable $x$.
Multiplying the identity \eqref{4.10} by $Y^n$ and by $Z^n$, we get respectively\[
Y^n=p_nY^{2n+1}+q_nY^nZ^{n+1},\qquad Z^n=p_nY^{n+1}Z^n+q_nZ^{2n+1},
\]
and this gives
\begin{align*}
&YZ\left(p_0q_nZ^n+q_0p_nY^n\right) \\
&\qquad =YZ\left[p_0q_n\left(p_nY^{n+1}Z^n+q_nZ^{2n+1}\right)
+q_0p_n\left(p_nY^{2n+1}+q_nY^nZ^{n+1}\right)\right] \\
&\qquad =\left(p_0p_nq_nZ^n+q_0p_n^2Y^n\right)ZY^{n+2}
+\left(p_0q_n^2Z^n+q_0p_nq_nY^n\right)YZ^{n+2}.
\end{align*}
Substituting this into \eqref{4.12} and showing the variable $x$ again, we get
\[
p_{n+1}(x)Y(x)^{n+2} + q_{n+1}(x)Z(x)^{n+2} = 1,
\]
where
\begin{align}
p_{n+1}(x) &= p_n(x)\left(p_0(x)+p_0(x)q_n(x)Z(x)^{n+1}
+q_0(x)p_n(x)Y(x)^nZ(x)\right),\label{4.13}\\
q_{n+1}(x) &= q_n(x)\left(q_0(x)+q_0(x)p_n(x)Y(x)^{n+1}
+p_0(x)q_n(x)Y(x)Z(x)^n\right).\label{4.14}
\end{align}
Since all polynomials on the right-hand sides of \eqref{4.13} and \eqref{4.14}
have integer coefficients, then so do $p_{n+1}$ and $q_{n+1}$. This proves our
result by induction.
\end{proof}

\begin{rem}\label{rem:4.3}
{\rm The two identities \eqref{4.13} and \eqref{4.14} allow us to construct
sequences of polynomials $p_n,q_n\in{\mathbb Z}[x]$, given an initial pair
$p_0,q_0$ satisfying \eqref{4.11}. However, this construction is not optimal
in the sense that the degrees of $p_n, q_n$ are substantially larger than
$n\deg{Z}$ and $n\deg{Y}$, respectively.

To compare this situation with the polynomials $P_n, Q_n$ defined in
\eqref{1.3}, we set $p_0(x)=-1$, $q_0(x)=1$, $Y(x)=x$, and $Z(x)=x+1$. Then
\eqref{4.11} is satisfied, and from \eqref{4.13}, \eqref{4.14} we get
$p_1(x)=2x+3=P_1(x)$ and $q_1(x)=-2x+1=Q_1(x)$; see also Table~1. However, we
further obtain
\begin{align*}
p_2(x)&=(2x+3)(4x^3+8x^2+3x-2)=8x^4+28x^3+30x^2+5x-6,\\
q_2(x)&=(-2x+1)(4x^3+4x^2-x+1)=-8x^4-4x^3+6x^2-3x+1,
\end{align*}
and the degrees of the next pairs of polynomials are 11, 26, 57, 120, $\ldots$,
which is sequence A000295 in \cite{OEIS}. It can be shown by induction that
for $n\geq 0$ we have
\[
\deg{p_n} = \deg{q_n} = 2^{n+1}-n-2.
\]
As is mentioned in \cite{OEIS}, these are specific Eulerian numbers, with
numerous other interesting properties.}
\end{rem}




\begin{rem}\label{rem:4.5}
{\rm One can go one step further and look for positive integers $n, m$ and polynomials $Y, Z\in\mathbb{C}[x]^{2}$ satisfying the general equation
\begin{equation}\label{Thue}
P_{n}(x)Y(x)^{m}+Q_{n}(x)Z(x)^{m}=1.
\end{equation}
We were unable to solve this equation. However, we can show that \eqref{Thue}
has infinitely many solutions. Indeed, if $m\geq 2$ and $n=rm-1$, then
$Y(x)=\zeta x^{r}, Z(x)=\xi (x+1)^{r}$ are solutions of \eqref{Thue}, where
$\zeta, \xi$ are arbitrary $m$th roots of unity. Here we have that $n>m$, and
this is no coincidence. Indeed, let us assume $m\geq 3$ and  differentiate
\eqref{Thue} with respect to $x$. Then we see that the polynomials $Y, Z$
satisfy the conditions
$$
Z(x)^{m-1}\mid P_{n}'(x)Y(x)+mP_{n}(x)Y'(x), \qquad
Y(x)^{m-1}\mid Q_{n}'(x)Z(x)+mQ_{n}(x)V'(x).
$$
As a consequence we get the inequalities
$$
(m-1)\op{deg}Z\leq n+\op{deg}Y-1\quad \mbox{and}
\quad (m-1)\op{deg}Y\leq n+\op{deg}Z-1,
$$
or equivalently
$$
\op{deg}Y=\op{deg}Z\leq \frac{(m-1)n}{m(m-2)}.
$$
If $n\leq m-2$ then $Y, Z$ are constant polynomials, and so for $n\geq 2$
there are no solutions. If $n=m-1$, then the equation \eqref{Thue} reduces to
\eqref{3.1}, which was already solved. Thus, $n\geq m$ remains to be considered.

The case $m=2$ leads to polynomial Pell equation
$P_{n}(x)Y(x)^2+Q_{n}(x)Z(x)^2=1$, and for $n\geq 3$ we believe that the only
solutions of this equation are of the form $n=2r-1, Y(x)=\pm x^{r},
Z(x)=\pm (x+1)^{r}$.
}
\end{rem}

\section{The discriminant of $Q_n^{(k)}(x)$}

One of the most important invariants of a polynomial is its discriminant. Since
we already noted certain similarities between our polynomials $P_n(x), Q_n(x)$
and the Chebyshev polynomials, we mention the fairly recent papers
\cite{DS1, GI, St, Tr1, Tr2} which have dealt with discriminants of
Chebyshev-like polynomials.

We begin with recalling the definition of the discriminant. Given a polynomial
$f(x)=a_mx^m+\cdots+a_1x+a_0\in{\mathbb Q}[x]$ with $a_m\neq 0$, the
discriminant of $f$ is usually defined by
\begin{equation}\label{5.1}
{\rm Disc}(f)=(-1)^{\frac{m(m-1)}{2}}a_{m}^{-1}R(f,f'),
\end{equation}
where $f'$ is the derivative of $f$ and $R(f,f')$ is the resultant of $f$ and
$f'$ which, among other equivalent definitions, can be given by
\begin{equation}\label{5.2}
R(f,f')=a_{m}^{m-1}\prod_{i=1}^{m}f'(\theta_{i}).
\end{equation}
Here $\theta_i\in{\mathbb C}$ is the $i$th root of $f(x)$. It follows from
\eqref{5.1} and \eqref{5.2} that ${\rm Disc}(f)=0$ if and only if $f$ has
multiple roots, which is a key property of the discriminant. An alternative
definition of the resultant involves a determinant (the {\it Sylvester
determinant}) whose entries consist only of the coefficients of the two
polynomials involved. This implies that $R(f,f')\in{\mathbb Z}$ if
$f\in{\mathbb Z}[x]$.

The following is the main result of this section. As in Section~2, we let
$Q_n^{(k)}(x)$ denote the $k$th derivative of $Q_n(x)$, with
$Q_n^{(0)}(x)=Q_n(x)$.

\begin{theorem}\label{thm:5.1}
For all integers $n>k\geq 0$ we have
\begin{equation}\label{5.3}
{\rm Disc}(Q_n^{(k)}(x))=(-1)^\varepsilon\frac{n+k+1}{\binom{2n}{n+k}}
\left(\frac{(n+k)!}{(n-k)!}\binom{2n}{n}(2n+1)\right)^{n-k-1},
\end{equation}
where $\varepsilon:=(n-k)(n-k-1)/2$. In particular, for $n\geq 1$,
\begin{equation}\label{5.4}
{\rm Disc}(Q_n(x))=(-1)^\frac{n(n-1)}{2}(n+1)(2n+1)^{n-1}\binom{2n}{n}^{n-2}.
\end{equation}
\end{theorem}

\begin{proof}
Let $\psi_{k,i}, 1\leq i\leq n-k$, be the roots of the polynomial
$Q_n^{(k)}(x)$. Then we obtain from \eqref{2.8a},
\begin{equation}\label{5.5}
Q_n^{(k+1)}(\psi_{k,i}) = (-1)^n\frac{(2n+1)!}{(n-k)!n!}
\cdot\frac{\psi_{k,i}^{n-k}}{1+\psi_{k,i}}\quad(1\leq i\leq n-k).
\end{equation}
With \eqref{2.3} we see that the leading coefficient of $Q_n^{(k)}(x)$ is
\begin{equation}\label{5.6}
(-1)^n\frac{n!}{(n-k)!}\binom{2n}{n} = (-1)^n\frac{(2n)!}{(n-k)!n!},
\end{equation}
and thus by \eqref{5.2} we have
\begin{align}
R(Q_n^{(k)},Q_n^{(k+1)})
&= (-1)^{n(n-k-1)}\left(\frac{(2n)!}{(n-k)!n!}\right)^{n-k-1}
\prod_{i=1}^{n-k}Q_{n}^{(k+1)}(\psi_{k,i})\label{5.7}\\
&= (-1)^n(2n+1)^{n-k}\left(\frac{(2n)!}{(n-k)!n!}\right)^{2(n-k)-1}
\prod_{i=1}^{n-k}\frac{\psi_{k,i}^{n-k}}{1+\psi_{k,i}},\nonumber
\end{align}
where we have used \eqref{5.5} in the second equation. The leading
coefficient \eqref{5.6} means that we can write
\begin{equation}\label{5.8}
Q_n^{(k)}(x) = (-1)^n\frac{(2n)!}{(n-k)!n!}\prod_{i=1}^{n-k}(x-\psi_{k,i}).
\end{equation}
First we set $x=0$ in \eqref{5.8} and use the identity
$Q_n^{(k)}(0)=(-1)^k(n+k)!/n!$, which follows immediately from \eqref{2.3}.
This gives
\begin{equation}\label{5.9}
\prod_{i=1}^{n-k}\psi_{k,i} = \frac{(n-k)!(n+k)!}{(2n)!}.
\end{equation}
Next, with $x=-1$ in \eqref{5.8} and using the identity \eqref{2.8b}, we get
\begin{equation}\label{5.10}
\prod_{i=1}^{n-k}\left(1+\psi_{k,i}\right) = \frac{2n+1}{n+k-1}.
\end{equation}
Combining \eqref{5.9} and \eqref{5.10} with \eqref{5.7}, we obtain
\[
R(Q_n^{(k)},Q_n^{(k+1)}) = (-1)^n(2n+1)^{n-k}\frac{n+k+1}{2n+1}\cdot
\frac{(n+k)!}{n!}\left(\frac{(n+k)!(2n)!}{(n-k)!n!^2}\right)^{n-k-1}.
\]
This, with \eqref{5.1} and the coefficient \eqref{5.6}, gives us
\[
{\rm Disc}(Q_n^{(k)}(x))=(-1)^\varepsilon\frac{(n+k+1)!(n-k)!}{(2n+1)!}
(2n+1)^{n-k}\left(\frac{(n+k)!}{(n-k)!}\binom{2n}{n}\right)^{n-k-1},
\]
with $\varepsilon$ as defined in the statement of the theorem. This last
identity immediately implies \eqref{5.2}, and by setting $k=0$ we get
\eqref{5.4}.
\end{proof}

We briefly turn our attention to the companion polynomial of $Q_n(x)$, namely
$P_n(x)$, and prove the following result.

\begin{corollary}\label{cor:5.2}
For all integers $n>k\geq 0$ we have
\begin{equation}\label{5.11}
{\rm Disc}(P_n^{(k)}(x)) = {\rm Disc}(Q_n^{(k)}(x)).
\end{equation}
\end{corollary}

\begin{proof}
For a polynomial $f$ of degree $n$ it was shown in \cite[Lemma~4.3]{DS1} that
\[
{\rm Disc}(f(ax+b)) = a^{n(n-1)}{\rm Disc}(f(x)),\quad
{\rm Disc}(cf(x)) = c^{2(n-1)}{\rm Disc}(f(x)),
\]
where $a, b, c$ are constants. We apply these identities to
\[
P_n^{(k)}(x) = (-1)^{n+k+1}Q_n^{(k)}(-x-1),
\]
which follows from \eqref{2.1}. Hence we have $a=-1$ and $c=\pm 1$, and since
both exponents are even, we immediately get \eqref{5.11}.
\end{proof}

\begin{rem}\label{rem:5.3}
{\rm Since we have used the Chebyshev polynomials of both kinds earlier in this
paper, it is worth mentioning their discriminants:
\[
{\rm Disc}(T_n(x)) = 2^{(n-1)^2}n^n,\qquad
{\rm Disc}(U_n(x)) = 2^{n^2}(n+1)^{n-2};
\]
see, e.g., \cite{Tr2}.}
\end{rem}

Knowledge of the discriminant of a polynomial is often important for determining
the polynomial's Galois group. In particular, it is known that if the
discriminant is the square of a nonzero integer, then the Galois group is a
subgroup of the alternating group $A_n$; see, e.g., \cite[p.~??]{DF}. It is
therefore of interest to find square discriminants.

\begin{corollary}\label{cor:5.4}
Let $n>k\geq 0$ be integers, and set $D_{k,n}:={\rm Disc}(Q_n^{(k)})$.\\
$(a)$ If $n\equiv k+2$ or $k+3\pmod{4}$, then $D_{k,n}$ is not the square
of an integer.\\
$(b)$ If $n\equiv k+1\pmod{4}$, then for a given $k$, $D_{k,n}$ is a square
for at most finitely many $n$.\\
$(c)$ If $n\equiv k\pmod{4}$, then for each $k$ there are infinitely many
$n$ such that $D_{k,n}$ is a square.\\
$(d)$ In particular, $D_{0,n}$ is a square if and only if $n=1$ or
$n=n_j$, where
\begin{equation}\label{5.12}
n_j:=\frac{1}{8}\left((3+2\sqrt{2})^{2j+1}+(3-2\sqrt{2})^{2j+1}-6\right),\qquad
j=1, 2, 3, \ldots
\end{equation}
\end{corollary}

\begin{proof}
(a) If $n-k\equiv 2$ or $3\pmod{4}$, then clearly $\varepsilon=(n-k)(n-k-1)/2$
is odd, and thus $D_{k,n}<0$, which cannot be a square.

(b) If $n-k\equiv 1\pmod{4}$, then with \eqref{5.3} we see that $D_{k,n}$ is a
square if and only if
\begin{equation}\label{5.13}
\frac{1}{n+k+1}\binom{2n}{n+k}=\frac{(2n)(2n-1)\cdots(n+k+2)}{(n-k)!}
\end{equation}
is a square. However, the prime number theorem implies that for a fixed $k$ and
for $n$ sufficiently large, there is always a prime among the members of the
sequence $n+k+2, n+k+3,\ldots, 2n-1$; this means that \eqref{5.13} cannot be a
square for these $k$ and $n$, which proves part (b).

(c) If $n-k\equiv 0\pmod{4}$, then $n-k-2$ is even, and according to \eqref{5.3}
we consider
\[
\frac{n+k+1}{\binom{2n}{n+k}}\frac{(n+k)!}{(n-k)!}\binom{2n}{n}(2n+1)
= (n+k+1)(2n+1)\left(\frac{(n+k)!}{n!}\right)^2,
\]
where the equality is easily seen by writing the binomial coefficients on the
left in terms of factorials. This now implies that $D_{k,n}$ is a square if and
only if $(n+k+1)(2n+1)$ is s square. If we set $n=4m+k$, where $m$ is a
positive integer, we can write
\[
(4m+2k+1)(8m+2k+1) = Y^2
\]
for some integer $Y\geq 1$. We can write this in the equivalent form
\begin{equation}\label{5.14}
X^2 - 8Y^2 = (2k+1)^2,
\end{equation}
where $X=16m+6k+3$. The equation \eqref{5.14} is a Pell-type equation, and
using standard methods for solving such equations (see, e.g.,
\cite[Sect.~7.8]{NZM}), we find that $X=X_j$, where $X_1=3(2k+1)$,
$X_2=17(2k+1)$, and $X_j=6X_{j-1}-X_{j-2}$ for $j\geq 3$. From the theory of
linear recurrence relations we get the Binet-type formula
\begin{equation}\label{5.15}
X_j = \frac{2k+1}{2}\left((3+2\sqrt{2})^j+(3-2\sqrt{2})^j\right).
\end{equation}
We require $X_j\equiv 6k+3\pmod{16}$, and an easy calculation shows that this
holds if and only if $j$ is odd. The corresponding $n$ that makes $D_{k,n}$
square is then
\[
n = n_j = \frac{1}{2}\left(X_{2j+1}-2k-3\right),
\]
which proves part (c).

(d) For $k=0$, together with \eqref{5.15}, this last identity gives
\eqref{5.12}. To complete the proof, we first note that in the case $k=0$ the
theory of Pell equations tells us that {\it all} solutions of \eqref{5.14} are
given by \eqref{5.15}. Returning to part (b) in the case $k=0$, we see that by
\eqref{5.4} we have $D_{0,1}=1$. Finally, considering again \eqref{5.13},
there is a prime in the sequence $n+2, n+3,\ldots, 2n-1$; this follows by
Bertrand's Postulate (itself a consequence of the Prime Number Theorem) and
the fact that in this case $n+1$ and $2n$ cannot be prime. The proof is now
complete.
\end{proof}

\begin{rem}\label{rem:5.5}
{\rm (a) The first few values of $n>1$ making $D_{0,n}$ a square are}
$$
n=24, 840, 28560, 970224, 32959080, 1119638520, 38034750624,\ldots .
$$
{\rm (b) The statement of part (b) of Corollary~\ref{cor:5.4} can be made more
explicit by using the following result of Nagura \cite{Na}:} For $x\geq 25$
there is always a prime $p$ with $x<p<\frac{6}{5}x$.

{\rm Setting $x=n+k+1$, an easy calculation shows that whenever
$n>\frac{3}{2}(k+1)$, at least one of the integers $n+k+2,\ldots, 2n-1$ is a
prime. This implies that $D_{k,n}$ is not a square when $n\equiv k+1\pmod{4}$
and $n>\frac{3}{2}(k+1)$. The small cases corresponding to $x<25$ can be
eliminated by direct computation. In fact, using the above limit, we verified
computationally that $D_{k,n}$ is not a square for any $0\leq k\leq 5000$ and
$n\equiv k+1\pmod{4}$. The constant $3/2$ in the above bound could be lowered
by using later improvements of Nagura's result, however at the cost of a higher
bound than Nagura's $x\geq 25$.}
\end{rem}

\section{Resultants involving $Q_{n+k}^{(k)}(x)$}

In the previous section we determined the discriminants, that is, resultants
of a polynomial and its derivative, We will now see that the resultant of two
consecutive polynomials in our sequence $Q_n(x)$ has a particularly easy form.

\begin{theorem}\label{thm:6.1}
For any integer $n\geq 1$ we have
\begin{equation}\label{6.1}
R(Q_n(x),Q_{n-1}(x)) = 2^n\binom{2n}{n}^{n-2}.
\end{equation}
\end{theorem}

This result is a consequence of the following theorem.
For greater ease of notation, we set
\[
\Delta_{k,n}:=R(Q_{n+k}^{(k)}(x),Q_{n-1+k}^{(k)}(x)).
\]

\begin{theorem}\label{thm:6.1b}
Let $k\geq 0$ be a fixed integer. Then we have
\begin{equation}\label{6.1b}
\Delta_{k,n} = \frac{2(2k+n)}{n}\left(\frac{(2k+2n)!}{n!(k+n)!}\right)^{n-2}
\left(\frac{(2k+n-1)!}{(k+n-1)!}\right)^nQ_{n-1+k}^{(k)}(-\tfrac{n+2k}{2(n+k)}).
\end{equation}
\end{theorem}

For $k=0$, the identity \eqref{6.1b} reduces to
\[
\Delta_{0,n} = 2\binom{2n}{n}^{n-2}Q_{n-1}(-\tfrac{1}{2})
= 2^n\binom{2n}{n}^{n-2},
\]
where we have used Corollary~\ref{cor:2.2} for the second equation. We have
thus shown that Theorem~\ref{thm:6.1} follows from Theorem~\ref{thm:6.1b}.

Before proving Theorem~\ref{thm:6.1b}, we summarize some useful properties of
the resultant. See \cite[Sect.~4]{DS1} for these and more, including references.
Suppose we have the two polynomials
\begin{equation}\label{6.1a}
\begin{cases}
f(x)&=a_0x^\mu+\cdots+a_{\mu-1}+a_\mu
= a_0(x-\alpha_1)\cdots(x-\alpha_\mu),\\
g(x)&=b_0x^m+\cdots+a_{m-1}+a_m
= b_0(x-\beta_1)\cdots(x-\beta_m).
\end{cases}
\end{equation}
Assuming the variable $x$, and that the resultant is taken with respect to $x$,
the following properties hold:
\begin{align}
R(f,g) &= a_0^m\prod_{i=1}^\mu g(\alpha_i),\label{6.2}\\
R(f,g) &= (-1)^{\mu m}R(g,f),\label{6.3}\\
R(f,pq) &= R(f,p)\cdot R(f,q),\label{6.4}
\end{align}
where $p$ and $q$ are arbitrary polynomials in $x$. Next, if $a$ is a
constant and $g$ is a polynomial then, unless $a=g=0$, we have
\begin{equation}\label{6.5}
R(a,g) = R(g,a) = a^{\deg{g}}.
\end{equation}

Finally, we require the following lemma, which can be found as Lemma~4.1 in
\cite{DS1} or, with a different normalization, in \cite[p.~58]{PZ}.

\begin{lemma}\label{lem:6.2}
Let $f, g$ be as in \eqref{6.1a}. If we can write
\[
f(x) = q(x)g(x) + r(x)
\]
with polynomials $q, r$ and $\nu:=\deg{r}$, then
\begin{equation}\label{6.6}
R(g,f) = b_0^{\mu-\nu}R(g,r).
\end{equation}
\end{lemma}

\begin{proof}[Proof of Theorem~\ref{thm:6.1b}]
We fix $k\geq 0$, and to simplify notation we set
$$
\overline{Q}_n(x) := Q_{n+k}^{(k)}(x)\quad\hbox{and}\quad
\Delta_{n} := \Delta_{k,n}.
$$
Our first goal is to find a recurrence relation between $\Delta_{n}$ and
$\Delta_{n-1}$, using the recurrence relation \eqref{2.11}. For greater ease of
notation, we set
$u_{n}(x):=u_{k,n}(x)$, $v_{n}(x):=v_{k,n}(x)$, and $w_{n}(x):=w_{k,n}(x)$,
and suppress the variable $x$ when there is no danger of confusion.

We apply \eqref{6.5} and \eqref{6.6} with
\[
f(x)=u_{n}(x)\overline{Q}_n(x),\qquad g(x)=\overline{Q}_{n-1}(x),\qquad
r(x)=w_{n}(x)\overline{Q}_{n-2}(x),
\]
and note that $\mu=n+2, \nu=n$. With the explicit expansion \eqref{2.11b} we get
\begin{equation}\label{6.7}
b_0^{\mu-\nu} = \frac{(n-1+k)!^2}{(n-1)!}^2\binom{2(n-1+k)}{n-1+k}^2=:r_n.
\end{equation}
Then the identity \eqref{6.6} applies to \eqref{2.11} gives us
\begin{equation}\label{6.8}
R(\overline{Q}_{n-1},u_{n}\overline{Q}_n)
= r_nR(\overline{Q}_{n-1},w_{n}\overline{Q}_{n-2}).
\end{equation}
Now with \eqref{6.4} we get
\begin{align}
R(\overline{Q}_{n-1},u_n\overline{Q}_n) &= R(\overline{Q}_{n-1},u_n)
R(\overline{Q}_{n-1},\overline{Q}_n)=R(\overline{Q}_{n-1},u_{n})\Delta_{n},
\label{6.9}\\
R(\overline{Q}_{n-1},w_{n}\overline{Q}_{n-2}) &= R(\overline{Q}_{n-1},w_{n})
R(\overline{Q}_{n-1},\overline{Q}_{n-2})=R(\overline{Q}_{n-1},w_{n})
\Delta_{n-1},\label{6.10}
\end{align}
and the identities \eqref{6.6} and \eqref{6.2} yield, along with the identities
in Proposition~\ref{prop:2.9},
\begin{align*}
R&(\overline{Q}_{n-1},u_{n}) = (-1)^{n-1}R(u_{n},\overline{Q}_{n-1})\\
&=(-1)^{n-1}(n(n+k))^{n-1}R(x+1,\overline{Q}_{n-1})
R(2(n+k-1)x+n+2k-1,\overline{Q}_{n-1})\\
&= (-1)^{n-1}(n(n+k))^{n-1}\overline{Q}_{n-1}(-1)(2(n+k-1))^{n-1}
\overline{Q}_{n-1}(\alpha_n),
\end{align*}
where
\begin{equation}\label{6.11}
\alpha_n := \frac{1-n-2k}{2(n+k-1)}.
\end{equation}
Using Corollary~\ref{cor:2.7} to evaluate $\overline{Q}_{n-1}(-1)$, we then get
\begin{equation}\label{6.12}
R(\overline{Q}_{n-1},u_{n}) = (-1)^{n+k-1}p_n\cdot\overline{Q}_{n-1}(\alpha_n),
\end{equation}
where
\begin{equation}\label{6.13}
p_n=\frac{\big(2n(n+k-1)(n+k)\big)^{n-1}(2n+2k-1)!}{(n-1)!(n+2k)(n+k-1)!}.
\end{equation}
Similarly, we get
\begin{align*}
R&(\overline{Q}_{n-1},w_{n}) = (-1)^{n-1}R(w_{n},\overline{Q}_{n-1})\\
&= \big(-2(2n+2k-1)(n+2k-1)\big)^{n-1}R(x,\overline{Q}_{n-1})
R((2n+k)x+n+2k,\overline{Q}_{n-1}),\\
&=\big(2(2n+2k-1)(n+2k-1)\big)^{n-1}\overline{Q}_{n-1}(0)(2(n+k))^{n-1}
\overline{Q}_{n-1}(\beta_{n}),
\end{align*}
where
\begin{equation}\label{6.14}
\beta_n := -\frac{n+2k}{2(n+k)}.
\end{equation}
Using \eqref{2.11b} to evaluate $\overline{Q}_{n-1}(0)$, we get
\begin{equation}\label{6.15}
R(\overline{Q}_{n-1},w_{n}) = (-1)^{n+k-1}q_n\cdot\overline{Q}_{n-1}(\beta_n),
\end{equation}
where
\begin{equation}\label{6.16}
q_n=\big(4(n+k)(2n+2k-1)(n+2k-1)\big)^{n-1}\frac{(n+2k-1)!}{(n+k-1)!}.
\end{equation}
Next we multiply both sides of \eqref{6.12} by $\Delta_n$ and combine it with
\eqref{6.9} and \eqref{6.8}, obtaining
\begin{equation}\label{6.17}
(-1)^{n+k-1}\Delta_np_n\overline{Q}_{n-1}(\alpha_n)
= r_nR(\overline{Q}_{n-1},w_{n}\overline{Q}_{n-2}).
\end{equation}
Similarly, multiplying both sides of \eqref{6.15} by $\Delta_{n-1}$ and
using \eqref{6.10}, we get
\begin{equation}\label{6.18}
(-1)^{n+k-1}\Delta_{n-1}q_n\overline{Q}_{n-1}(\beta_n)
= R(\overline{Q}_{n-1},w_{n}\overline{Q}_{n-2}).
\end{equation}
Finally, combining \eqref{6.17} and \eqref{6.18}, we get the desired recurrence
relation
\begin{equation}\label{6.19}
\Delta_{n} = \frac{q_nr_n}{p_n}\cdot
\frac{\overline{Q}_{n-1}(\beta_n)}{\overline{Q}_{n-1}(\alpha_n)}\Delta_{n-1}
\qquad (n\geq 2),
\end{equation}
with
\[
\Delta_{1} = R(\overline{Q}_{1},\overline{Q}_{0})
=R((-1)^k\frac{(2k+1)!}{(k+1)!}\big(1-2(k+1)x\big),(-1)^k\frac{(2k)!}{k!}),
\]
where we have used \eqref{2.11b}. By \eqref{6.5}, this last resultant evaluates
to
\begin{equation}\label{6.20}
\Delta_1 = (-1)^k\frac{(2k)!}{k!}.
\end{equation}
This completes the first part of the proof.

For the second part of the proof we iterate \eqref{6.19} and combine it with
\eqref{6.20}, obtaining
\begin{equation}\label{6.21}
\Delta_{n} = (-1)^k\frac{(2k)!}{k!}\cdot
\left(\prod_{j=2}^n\frac{q_jr_j}{p_j}\right)\cdot
\left(\prod_{j=2}^n\frac{\overline{Q}_{j-1}(\beta_j)}{\overline{Q}_{j-1}(\alpha_j)}\right).
\end{equation}
We first deal with the second product in \eqref{6.21}. Comparing \eqref{6.11}
with \eqref{6.14}, we see that $\alpha_n=\beta_{n-1}$ for all $n\geq 2$. Hence
\begin{equation}\label{6.22}
\prod_{j=2}^n\frac{\overline{Q}_{j-1}(\beta_j)}{\overline{Q}_{j-1}(\alpha_j)}
=\prod_{j=2}^n\frac{\overline{Q}_{j-1}(\beta_j)}{\overline{Q}_{j-1}(\beta_{j-1})}
=\frac{\overline{Q}_{n-1}(\beta_n)}{\overline{Q}_1(\beta_1)}\cdot
\prod_{j=2}^{n-1}\frac{\overline{Q}_{j-1}(\beta_j)}{\overline{Q}_j(\beta_j)}.
\end{equation}
Since $\beta_1=-(2k+1)/2(k+1)$, we easily find with \eqref{2.11b} that
\begin{equation}\label{6.23}
\overline{Q}_1(\beta_1) = (-1)^k\frac{(2k+2)!}{(k+1)!}.
\end{equation}
Next, the terms in Proposition~\ref{prop:2.9} can be shown to evaluate as
\[
u_n(\beta_n)=\frac{n^2k}{2(n+k)},\qquad
v_n(\beta_n)=\frac{nk(n+2k)}{n+k},\qquad w_n(\beta_n)=0,
\]
and so the recurrence relation \eqref{2.11} gives
\[
\frac{\overline{Q}_{n-1}(\beta_n)}{\overline{Q}_{n-1}(\beta_{n-1})}
=\frac{n}{2(n+2k)}.
\]
This, with \eqref{6.22} and \eqref{6.23}, gives
\[
\prod_{j=2}^n\frac{\overline{Q}_{j-1}(\beta_j)}{\overline{Q}_{j-1}(\alpha_j)}
= (-1)^k\frac{(k+1)!}{(2k+2)!}\overline{Q}_{n-1}(\beta_n)
\prod_{j=2}^{n-1}\frac{j}{2(j+2k)},
\]
and with the straightforward evaluation
\[
\prod_{j=2}^{n-1}\frac{j}{2(j+2k)} = \frac{(n-1)!(2k+1)!}{2^{n-2}(2k+n-1)!}
\]
we get
\begin{equation}\label{6.24}
\prod_{j=2}^n\frac{\overline{Q}_{j-1}(\beta_j)}{\overline{Q}_{j-1}(\alpha_j)}
= (-1)^k\frac{(n-1)!k!}{2^{n-1}(2k+n-1)!}\overline{Q}_{n-1}(\beta_n).
\end{equation}

Next we deal with the first product in \eqref{6.21}. We denote it by $\Pi_n$,
and by combining \eqref{6.7}, \eqref{6.13} and \eqref{6.16} we get
\begin{equation}\label{6.25}
\Pi_n=\prod_{j=2}^n 2^{j-1}\frac{(2j+2k-2)!(j+2k)!}{j!(j+k-1)!^2}\cdot
\frac{(2j+2k-1)^{j-2}(j+2k-1)^{j-1}}{j^{j-2}(j+k-1)^{j-1}}.
\end{equation}
To evaluate this product, we first consider
\[
\prod_{j=2}^n j^{j-2} = 2^0\cdot 3^1\cdot 4^2\cdots n^{n-2}
=\frac{n!}{2!}\cdot\frac{n!}{3!}\cdot\frac{n!}{4!}\cdots\frac{n!}{(n-1)!},
\]
so that
\begin{equation}\label{6.26}
\prod_{j=2}^n j^{j-2} = (n!)^{n-2}\prod_{j=2}^n\frac{1}{(j-1)!}.
\end{equation}
Similarly we obtain
\begin{align}
\prod_{j=2}^n(j+k-1)^{j-1} &= (k+n-1)!^{n-1}\prod_{j=2}^n\frac{1}{(j+k-2)!},\label{6.27}\\
\prod_{j=2}^n(j+2k-1)^{j-1} &= (2k+n-1)!^{n-1}\prod_{j=2}^n\frac{1}{(j+2k-2)!},\label{6.28}
\end{align}
and with some more effort,
\begin{equation}\label{6.29}
\prod_{j=2}^n(2j+2k-1)^{j-2} = \left(\frac{(2k+2n)!}{(k+n)!}\right)^{n-1}
\prod_{j=2}^n\frac{(j+k-1)!}{2^{j-1}(2j+2k-1)!}.
\end{equation}
Substituting \eqref{6.26}--\eqref{6.29} into \eqref{6.25} and simplifying, we
get
\begin{equation}\label{6.30}
\Pi_n = n!\left(\frac{(2k+2n)!(2k+n-1)!}{n!(k+n)!(k+n-1)!}\right)^{n-1}
\prod_{j=2}^n\frac{(j+2k)(j+2k-1)}{j(2j+2k-1)(j+k-1)}.
\end{equation}
It is clear that the product on the right of \eqref{6.30} can be written as a
product and quotient of factorials, and after working out the details, we get
\begin{equation}\label{6.31}
\Pi_n = 2^n\binom{2k+n}{n}\left(\frac{(2k+2n)!}{n!(k+n)!}\right)^{n-2}
\left(\frac{(2k+n-1)!}{(k+n-1)!}\right)^n.
\end{equation}
Finally, substituting \eqref{6.31} and \eqref{6.24} into \eqref{6.21} and
simplifying, we obtain \eqref{6.2}. The proof is now complete.
\end{proof}

As an immediate consequence of Theorem~\ref{thm:6.1b} we obtain the following
result.

\begin{corollary}\label{cor:6.3}
For any integers $k\geq 0$ and $n\geq 1$ we have $(-1)^k\Delta_{k,n}>0$.
\end{corollary}

\begin{proof}
From the explicit expansion \eqref{2.11b} we see that
$(-1)^kQ_{k+n-1}^{(k)}(x)>0$ whenever $x<0$. This, together with \eqref{6.1b},
implies the statement.
\end{proof}

Our next result, which is rather surprising, but is easy to prove, is similar
in nature to Theorem~\ref{thm:6.1}.

\begin{theorem}\label{thm:6.2}
For any $n\geq 0$ we have
\begin{equation}\label{6.13a}
R(P_n(x),Q_n(x)) = \binom{2n}{n}^{n+1}.
\end{equation}
\end{theorem}

\begin{proof}
We rewrite the defining equation \eqref{1.3} as
\[
P_n(x)x^{n+1} = -(x+1)^{n+1}Q_n(x) + 1,
\]
and apply Lemma~\ref{lem:6.2} with $f(x)=P_n(x)x^{n+1}$, $g(x)=Q_n(x)$, and
$r(x)=1$. Since $b_0=(-1)^n\binom{2n}{n}$, $\mu=2n+1$ and $\nu=0$, we get with
\eqref{6.6},
\begin{equation}\label{6.14a}
R(Q_n,P_nx^{n+1}) = (-1)^n\binom{2n}{n}^{n+1}R(Q_n,1).
\end{equation}
Now, by \eqref{6.5} we have $R(Q_n,1)=1$, and with \eqref{6.4}, \eqref{6.3} and
\eqref{6.2} we get
\begin{align*}
R(Q_n,P_nx^{n+1}) &= R(Q_n,P_n)\cdot R(Q_n,x^{n+1})\\
&= R(Q_n,P_n)\cdot R(x^{n+1},Q_n) = R(Q_n,P_n)\cdot 1.
\end{align*}
This, together with \eqref{6.14}, gives
\[
R(Q_n,P_n) = (-1)^n\binom{2n}{n}^{n+1},
\]
and \eqref{6.4} finally yields the desired identity \eqref{6.13}.
\end{proof}

\section{The polynomial $Q_{n}(-x)^{2}-Q_{n-1}(-x)Q_{n+1}(-x)$}

Returning to the occasional comparisons we made between the polynomials
$P_n(x)$, $Q_n(x)$ and the Chebyshev polynomials of both kinds, we recall the
well-known $2\times 2$ Hankel determinant expressions
\[
T_n(x)^2 - T_{n-1}(x)T_{n+1}(x) = 1-x^2,\qquad
U_n(x)^2 - U_{n-1}(x)U_{n+1}(x) = 1,
\]
valid for all integers $n\geq 1$; see, e.g., \cite[p.~40]{Ri}. This immediately
gives rise to the question of what can be said about analogous expressions for
the polynomials $Q_n(x)$. For reasons of simplicity we consider $Q_n(-x)$
instead; the first few $2\times 2$ Hankel determinant expressions, factored
over $\mathbb Q$, are listed in Table~3.

\bigskip
\begin{center}
{\renewcommand{\arraystretch}{1.2}
\begin{tabular}{|l|l|}
\hline
 $n$ & $Q_{n}(-x)^{2}-Q_{n-1}(-x)Q_{n+1}(-x)$\\
\hline
 1 & $x (1-2 x)$ \\
 2 & $x^2 (1-2 x) \left(2 x+3\right)$ \\
 3 & $5 x^3 (1-2 x) \left(2 x^2+3 x+2\right)$ \\
 4 & $7 x^4 (1-2 x) \left(10 x^3+15 x^2+12 x+5\right)$ \\
 5 & $42 x^5 (1-2 x) \left(14 x^4+21 x^3+18 x^2+10 x+3\right)$ \\
 6 & $66 x^6 (1-2 x) \left(84 x^5+126 x^4+112 x^3+70 x^2+30 x+7\right)$ \\
 7 & $429 x^7 (1-2 x) \left(132 x^6+198 x^5+180 x^4+120 x^3+60 x^2+21 x+4\right)$ \\
 \hline
\end{tabular}}

\medskip
{\bf Table 3.} Factorization of $Q_{n}(-x)^{2}-Q_{n-1}(-x)Q_{n+1}(-x)$, $n=1,\ldots, 7$.
\end{center}
\bigskip

The above table strongly suggest that $Q_{n}(-x)^{2}-Q_{n-1}(-x)Q_{n+1}(-x)$
is always divisible by $c_{n}x^{n}(1-2x)$, where $c_{n}\in\Z$ is a constant
depending on $n$, and the co-factor has degree $n-1$ and has positive
coefficients. Moreover, it also seems that the sequence of coefficients of the
co-factor forms a unimodal sequence. In this section we will show that these
properties are true in general. We start with the following.

\begin{proposition}\label{prop:7.1}
For integers $n\geq 1$ we define
\begin{equation}\label{7.1}
V_{n}(x):=\frac{2(n+1)}{\binom{2n}{n}x^{n}(1-2x)}
\left(Q_{n}(-x)^2-Q_{n-1}(-x)Q_{n+1}(-x)\right).
\end{equation}
Then the sequence of polynomials $V_{n}(x)$ satisfies the following recurrence
relation: $V_{1}(x)=2$, $V_{2}(x)=2x+3$, and for $n\geq 3$ we have
\begin{equation}\label{7.2}
n(x-1)V_{n}(x)=\big(2(2n-3)x(x-1)-n-1\big)V_{n-1}(x)+2(2n-1)xV_{n-2}(x).
\end{equation}
\end{proposition}

\begin{proof}
The recurrence relation was guessed with the help of {\tt EKHAD} (see
Remark \ref{rem:2.11}), which showed that the recurrence holds for all
$n\leq 25$. The proof of correctness for all $n$ is a straightforward but
tedious induction using the recurrence relation satisfied by the polynomials
$Q_{n}(-x)$. We omit the details.
\end{proof}

For integers $n\geq 1$ we now set
\begin{equation}\label{7.3}
V_{n}(x)=\sum_{k=0}^{n-1}a_{k,n}x^{k}.
\end{equation}
Then the sequences of the first few coefficients are readily identified as
$a_{0,n}=n+1$,
$a_{1,n}=(n-1)n$, $a_{2,n}=\frac{1}{2}(n-2)(n-1)(n+1)$,
$a_{3,n}=\frac{1}{6}(n-3)(n-2)(n+1)(n+2)$,
which are consistent with the entries in Table~3. These are special cases of
the following closed expression for the polynomials $V_n(x)$.

\begin{theorem}\label{thm:7.2}
For integers $n\geq 1$ we have
\begin{equation}\label{7.4}
V_{n}(x)=\sum_{k=0}^{n-1}\frac{(n-k)(n-k+1)}{n}\binom{n-1+k}{k}x^{k}.
\end{equation}
\end{theorem}

\begin{proof}
Following along the lines of the expressions for $a_{0,n},\ldots, a_{3,n}$
above, it is easy to obtain experimentally a few more cases and then conjecture
the form of the coefficients on the right-hand side of \eqref{7.4}. It can now
be verified in a tedious but straightforward way that the right-hand side of
\eqref{7.4} satisfies the recurrence relation \eqref{7.2}. Since \eqref{7.4}
also clearly gives $V_1(x)=2$ and $V_2(x)=2x+3$, this proves the theorem.
\end{proof}

Using Theorem~7.2, we can easily obtain some properties of the polynomials
$V_n(x)$. The first one is as follows.

\begin{corollary}\label{cor:7.3}
All polynomials $V_n(x)$, $n\geq 1$, have positive integer coefficients.
\end{corollary}

\begin{proof}
Positivity of the coefficients is clear from \eqref{7.4}. Now, using the
notation of \eqref{7.3}, we rewrite the coefficients as
\begin{equation}\label{7.5}
a_{k,n} = (n-k)(n-k+1)\frac{(n+k-1)(n+k-2)\cdots(n+1)}{(k-1)!k}.
\end{equation}
The fraction on the right of \eqref{7.5} is an integer except when $k\mid n$.
But in this case we have $k\mid(n-k)$, so that in either case $a_{k,n}$ is an
integer.
\end{proof}

\begin{rem}\label{rem:7.4}
{\rm In the above proof of Corollary~\ref{cor:7.3} we actually showed a bit
more, namely that $a_{k,n}/(n+1-k)$ is an integer for all $n\geq 1$ and
$0\leq k\leq n-1$.}
\end{rem}

As another easy consequence of Theorem~\ref{thm:7.2} and
Proposition~\ref{prop:7.1} we get some special values of $V_n(x)$.

\begin{corollary}\label{cor:7.5}
For all integers $n\geq 1$ we have
\[
V_n(0)=n+1,\qquad V_n(\tfrac{1}{2})=2^n,\qquad
V_n(1)=\frac{1}{n+2}\binom{2n+2}{n+1} = C_{n+1},
\]
where $C_n$ is the $n$th Catalan number.
\end{corollary}

\begin{proof}
The first identity follows immediately from \eqref{7.4}. For the second and
third identities we substitute $x=\frac{1}{2}$ and $x=1$ into \eqref{7.2},
getting the recurrence relations
\begin{align*}
nV_n(\tfrac{1}{2})&=(4n-1)V_{n-1}(\tfrac{1}{2})-2(2n-1)V_{n-2}(\tfrac{1}{2}),\\
0&=-(n+1)V_{n-1}(1)+2(2n-1)V_{n-2}(1),
\end{align*}
respectively. The two identities are then obtained by easy inductions.
\end{proof}

%
%

Next we prove the second observation we made following Table~3. Recall that a
polynomial $\sum_{i=0}^{n}a_{i}x^{i}$ is called {\it unimodal} if and only if
the sequence of its coefficients is unimodal, that is, if there is an integer
$m$ (called the {\it mode}) with $0\leq m\leq n$, such that
\[
a_{0}\leq a_{1}\leq \ldots a_{m-1}\leq a_{m}\geq a_{m+1}\geq \ldots \geq a_{n}.
\]
We now state and prove the following result.

\begin{theorem}\label{thm:7.6}
For every integer $n\geq 1$, the polynomial $V_{n}(x)$ is unimodal. More
precisely, with the notation \eqref{7.3} we have
\[
a_{0,n}<a_{1,n}<\ldots < a_{n-3,n}<a_{n-2,n}>a_{n-1,n}.
\]
\end{theorem}

\begin{proof}
Recall that $V_{1}(x)=2$ and $V_{2}(x)=3+2x$. While for $n=1$ there is nothing
to prove, the statement is clearly true for $n=2$. In what follows we therefore
assume that $n\geq 3$. First, using \eqref{7.3} and \eqref{7.4}, we get
\[
\frac{a_{n-2,n}}{a_{n-1,n}}
=\frac{\frac{6}{n}\binom{2n-3}{n-2}}{\frac{2}{n}\binom{2n-2}{n-1}}
=\frac{3}{2}>1,
\]
where the second equation is easy to verify. Next, for each $k$ with
$1\leq k\leq n-2$, we consider
\[
\frac{a_{k,n}}{a_{k-1,n}}
= \frac{(n-k)\binom{n-1+k}{k}}{(n-k+2)\binom{n-2+k}{k-1}}
=\frac{(n-k)(n+k-1)}{k(n-k+2)}.
\]
It remains to show that this quotient is greater than 1, which is equivalent to
$(n-k)(n+k-1)>k(n-k+2)$. After an easy manipulation we see that this, in turn,
is equivalent to $k(n+1)<n^2-n$. But this is true for all $1\leq k\leq n-2$
since the ``worst case" $k=n-2$ leads to the inequality $(n-2)(n+1)<n^2-n$,
which is clearly true.
Thus we have $a_{k-1,n}<a_{k,n}$ for all $1\leq k\leq n-2$ and $n\geq 3$, which
completes the proof.
\end{proof}

In analogy to Section~2, we can also obtain a generating function for the
sequence of polynomials $V_n(x)$.

\begin{theorem}\label{thm:7.7}
Let
\begin{equation}\label{7.6}
\mathcal{V}(x,t)=\sum_{n=1}^{\infty}V_{n}(x)t^{n}
\end{equation}
be the ordinary generating function for the sequence $(V_{n}(x))_n$. Then
\begin{equation}\label{7.7}
\mathcal{V}(x,t)=\frac{-2t^2+2(2-3x)t-1+2x+(1-2x)\sqrt{1-4xt}}{2(t+x-1)^{2}}.
\end{equation}
\end{theorem}

\begin{proof}
We use the same approach as in the proof of Proposition~\ref{prop:2.10}. Indeed,
differentiating both sides of \eqref{7.3} with respect to $t$, manipulating the
resulting series as we did in the proof of Proposition~\ref{prop:2.10} and
using the recurrence relation \eqref{7.2}, we get the differential equation
\[
t(t+x-1)(4tx-1)\mathcal{V}'(x,t)+2t(3xt-x^2+x-1)\mathcal{V}(x,t)+2t(3xt+x-1)=0.
\]
Using standard method of solving linear differential equations of degree 1 we
easily get the general solution
\[
\mathcal{V}(x,t)=\frac{-2t^2+2(2-3x)t-1+2x +2c_1 \sqrt{1-4tx}}{2(t+x-1)^2}.
\]
The initial condition $\mathcal{V}(x,0)=0$ leads to $c_{1}=\frac{1}{2}(1-2x)$,
which finally gives the desired solution \eqref{7.4}.
\end{proof}

We note that there are certain similarities between the generating functions
\eqref{2.12} and \eqref{7.7}. In fact, they are related through an identity
involving partial derivatives. It can be verified through direct computation.

\begin{lemma}\label{lem:7.8}
Let $\mathcal{R}(x,t):=\mathcal{Q}(-x,t)$. Then
\begin{equation}\label{7.8}
\frac{\partial\mathcal{V}}{\partial t}
=x^2\frac{\partial^{2}\mathcal{R}}{\partial x^{2}}
+2t\frac{\partial^{2}\mathcal{R}}{\partial t\partial x}
+t^2\frac{\partial^{2}\mathcal{R}}{\partial t^{2}}
+2x\frac{\partial \mathcal{R}}{\partial x}
+4t\frac{\partial \mathcal{R}}{\partial t}+2\mathcal{R}.
\end{equation}
\end{lemma}

\begin{rem}\label{rem:7.9}
{\rm The identity \eqref{7.8}, along with the explicit formula \eqref{2.3} for
$Q_n(x)$, can be used to give an alternative proof of the explicit formula
\eqref{7.4} for the polynomial $V_n(x)$. We leave the details to the interested
reader.}
\end{rem}

To conclude this section, we return to the original expression of the title,
which we denote by
\begin{equation}\label{7.9}
W_n(x) := Q_n(-x)^2-Q_{n-1}(-x)Q_{n+1}(-x) = \sum_{i=0}^{2n}w_{i,n}x^i.
\end{equation}
The first few of these polynomials are listed in Table~4.

\bigskip
\begin{center}
{\renewcommand{\arraystretch}{1.2}
\begin{tabular}{|l|l|}
\hline
 $n$ & $W_{n}(x)$\\
\hline
 1 & $x-2x^2$ \\
 2 & $3x^2-4x^3-4x^4$ \\
 3 & $10x^3-5x^4-20x^5-20x^6 $ \\
 4 & $35x^4+14x^5-63x^6-140x^7-140x^8 $ \\
 5 & $126x^5+168x^6-84x^7-630x^8-1176x^9-1176x^{10} $ \\
 6 & $462x^6+1056 x^7+660 x^8-1848x^9-6468x^{10}-11088x^{11}-11088x^{12}$\\
 \hline
\end{tabular}}

\medskip
{\bf Table 4.} The polynomials $W_{n}(x)$, $n=1,\ldots, 6$.
\end{center}

\bigskip
The following properties of the polynomials $W_n(x)$ are an immediate
consequence of \eqref{7.1} and Theorem~\ref{thm:7.2}.

\begin{corollary}\label{cor:7.10}
For all integers $n\geq 1$ and $0\leq j\leq n-1$ we have $w_{j,n}=0$, and
\begin{equation}\label{7.10}
w_{n,n}=\frac{1}{2}\binom{2n}{n},\qquad w_{2n-1,n}=w_{2n,n}=-2C_{n-1}C_n,
\end{equation}
where $C_n=\tfrac{1}{n+1}\binom{2n}{n}$ is the $n$th Catalan number.
\end{corollary}

To motivate the last result of this section, we consider the entry for $n=4$
in Table~4 and note that
\[
\frac{14}{2}-\frac{63}{4}-\frac{140}{8}-\frac{140}{16} = -35.
\]
This is actually no surprise since by \eqref{7.1} we have $W_n(\frac{1}{2})=0$
for all $n\geq 1$ and thus, by \eqref{7.9} we have
\begin{equation}\label{7.11}
\sum_{j=1}^n w_{n+j,n}2^{-j} = -w_{n,n}=-\frac{1}{2}\binom{2n}{n}\qquad
(n\geq 1).
\end{equation}
This identity is a special case of the following result.

\begin{proposition}\label{prop:7.11}
For integers $n\geq 1$ and $0\leq i\leq n-1$ we have
\begin{equation}\label{7.12}
\sum_{j=i+1}^n w_{n+j,n}2^{i-j}
= -\frac{(n-i)(n-i+1)}{2n(n+1)}\binom{2n}{n}\binom{n-1+i}{i}.
\end{equation}
\end{proposition}

Before proving this identity we note that in the two extreme cases we get
\eqref{7.11} when $i=0$, and the right-most equation in \eqref{7.10} when
$i=n-1$.

\begin{proof}[Proof of Proposition~\ref{prop:7.11}]
We first prove the identity
\begin{equation}\label{7.13}
\frac{W_n(x)}{x^n} = \sum_{i=0}^n w_{n+i,n}x^i
= (2x-1)\sum_{i=0}^{n-1}\bigg(\sum_{j=i+1}^n\frac{w_{n+j,n}}{2^{j-i}}\bigg)x^i,
\end{equation}
where the left equation comes from \eqref{7.9}. To do so, we denote the
right-most term of \eqref{7.13} by $R_n(x)$ and manipulate the double sum
as follows:
\begin{align*}
R_n(x)&=\sum_{i=0}^{n-1}\bigg(\sum_{j=i+1}^n\frac{w_{n+j,n}}{2^{j-i-1}}\bigg)x^{i+1}
-\sum_{i=0}^{n-1}\bigg(\sum_{j=i+1}^n\frac{w_{n+j,n}}{2^{j-i}}\bigg)x^i\\
&=\sum_{i=1}^n\bigg(\sum_{j=i}^n\frac{w_{n+j,n}}{2^{j-i}}\bigg)x^i
-\sum_{i=0}^{n-1}\bigg(\sum_{j=i+1}^n\frac{w_{n+j,n}}{2^{j-i}}\bigg)x^i\\
&=\sum_{i=0}^n\bigg(\sum_{j=i}^n\frac{w_{n+j,n}}{2^{j-i}}\bigg)x^i
-\sum_{j=0}^n\frac{w_{n+j,n}}{2^j}
-\sum_{i=0}^{n}\bigg(\sum_{j=i+1}^n\frac{w_{n+j,n}}{2^{j-i}}\bigg)x^i\\
&=\sum_{i=0}^n w_{n+i,n}x^i - \sum_{j=0}^n w_{n+j,n}(\tfrac{1}{2})^j,
\end{align*}
where we have combined the first and third sum from the second-last line.
The second sum in the last line then vanishes since $W_n(\tfrac{1}{2})=0$ for
all $n\geq 0$, and the proof of \eqref{7.13} is complete.

Now we rewrite \eqref{7.1}, with \eqref{7.9}, as
\[
\frac{1}{2(n+1)}\binom{2n}{n}V_n(x) = \frac{W_n(x)}{x^n(1-2x)}
= -\sum_{i=0}^{n-1}\bigg(\sum_{j=i+1}^n\frac{w_{n+j,n}}{2^{j-i}}\bigg)x^i.
\]
Finally, using \eqref{7.4} and equating coefficients of $x^i$,
$0\leq i\leq n-1$, we get \eqref{7.12}.
\end{proof}

\section{Some irreducibility results}

In this brief section we prove some irreducibility results for the main objects
of study in this paper, namely the polynomials $Q_n(x)$ and their derivatives,
and the polynomials $V_n(x)$.

\begin{theorem}\label{thm:8.1}
Let $n\geq 1$ and $0\leq k\leq n-1$ be integers. Then the polynomial
$Q_n^{(k)}(x)$ is irreducible over $\mathbb Q$ when $n+k+1$ is a prime, or
when $2n+1$ is a prime.
\end{theorem}

\begin{proof}
To prove the first statement, we show that $Q_{p-k-1}^{(k)}(x)$ is
$p$-Eisenstein. To do so, we rewrite the explicit expression \eqref{2.11b} as
\[
Q_n^{(k)}(x) = \frac{(-1)^k}{n!}\sum_{i=0}^{n-k}(-1)^i(n+k+i)!\frac{x^i}{i!},
\]
so the coefficients of $Q_{p-k-1}^{(k)}(x)$ are
\[
\frac{(n+k+i)!}{n!i!} = \frac{(p+i-1)!}{(p-k-1)!i!},\qquad
i=0,1,\ldots, p-2k-1.
\]
We note that for each positive index $i$ in this range the numerator of the
last fraction is divisible by $p$, but not by $p^2$, while the factorials
$(p-k-1)!$ and $i!$ in the denominator are not divisible by $p$. It is also
clear that the constant coefficient, $(p-1)!/(p-k-1)!$, is not divisible by $p$.
Hence the polynomial $Q_{p-k-1}^{(k)}(x)$ is $p$-Eisenstein.

For the second statement, we assume that $p:=2n+1$ is prime, and show that
$Q_n^{(k)}(-x-1)$ is $p$-Eisenstein, which would imply irreducibility of
$Q_n^{(k)}(x)$. For this purpose we combine \eqref{2.1} with \eqref{2.6}, and
upon taking the $k$th derivative we get
\[
Q_n^{(k)}(-x-1)=(-1)^(2n+1)\binom{2n}{n}
\sum_{i=k}^{n}\frac{1}{n+i+1}\cdot\frac{n!}{(i-k)!(n-1)!}x^{i-k}.
\]
We know that this polynomial has integer coefficients, and we see that for each
index $i$ with $k\leq i\leq n-1$ the corresponding coefficient is divisible by
the prime $p=2n+1$, but not by $p^2$. For $i=n$, however, we have cancellation,
and thus the leading coefficient is not divisible by $p$. Hence
$Q_n^{(k)}(-x-1)$ is $p$-Eisenstein, and the proof is complete.
\end{proof}

\begin{theorem}\label{thm:8.2}
The polynomial $V_n(x)$ is irreducible over $\mathbb Q$ if $2n+1$ is prime.
\end{theorem}

\begin{proof}
We set again $p:=2n+1$ and show that $V_n(x+1)$ is $p$-Eisenstein. To do so,
we use \eqref{7.4} and a binomial expansion, followed by changing the order
of summation:
\begin{align*}
V_n&(x+1)\\
&=\sum_{k=0}^{n-1}\frac{(n-k)(n-k+1)}{n}\binom{n-1+k}{k}
\sum_{j=0}^k\binom{k}{j}x^j\\
&= \sum_{j=0}^{n-1}\bigg(
\sum_{k=j}^{n-1}\frac{(n-k)(n-k+1)}{n}\binom{n-1+k}{k}\binom{k}{j}\bigg)x^j\\
&= \sum_{j=0}^{n-1}\bigg(\frac{(n-1+j)!}{n!j!}
\sum_{k=0}^{n-1-j}(n-k-j)(n-k-j+1)\binom{n-1+k+j}{k}\bigg)x^j,
\end{align*}
where the last line results from a straightforward manipulation of the binomial
coefficients in the previous line. The inner sum in this last line can be
evaluated by various means, including the function {\tt sum} in Maple which,
after some manipulations, gives
\[
V_n(x+1) = \sum_{j=0}^{n-1}\frac{2}{n+j}\binom{n+j}{n}\binom{2n+1}{n+j+2}x^j.
\]
As in the proof of the previous theorem we observe that the coefficients of
$V_n(x+1)$ are integers, and that for $0\leq j\leq n-2$ they are all divisible
by $p=2n+1$, but not by $p^2$, and that the coefficient of $x^{n-1}$ is not
divisible by $p$. Hence $V_n(x+1)$ is $p$-Eisenstein, as claimed.
\end{proof}


\section{Further remarks and conjectures}

In this final section we collect some further remarks and conjectures related
to the objects studied in this paper.

\medskip
{\bf 1.} In Section~5 we already mentioned the fact that a polynomial in
${\mathbb Z}[x]$ with a square discriminant has its Galois group contained in
the alternating group $A_n$. Can anything more be said about the Galois group
of $Q_n$? Computations support the following conjecture, where we set
$D_{0,n}={\rm Disc}(Q_n)$, as in Corollary~\ref{cor:5.4}.

\begin{conjecture}\label{conj:9.1}
For integers $n\geq 2$ we have
\[
{\rm Gal}(Q_n) = \begin{cases}
A_n & \hbox{if $D_{0,n}$ is a square};\\
S_n & \hbox{if $D_{0,n}$ is not a square}.
\end{cases}
\]
\end{conjecture}

{\bf 2.} In Section~8 we proved some partial irreducibility results. However,
computations indicate that much fore is true.

\begin{conjecture}\label{conj:9.2}
For all integers $n\geq 1$ and $0\leq k\leq n-1$ the polynomials
$Q_n^{(k)}(x)$ and $V_n(x)$ are irreducible over $\mathbb Q$.
\end{conjecture}

{\bf 3.} It is clear that Proposition \ref{prop:4.2} is in fact true in greater generality. More
precisely, if $R$ is a commutative ring with $0\neq 1$ and $p_{0}, q_{0}, Y, Z\in R$ satisfy the identity $p_{0}Y+q_{0}Z=1$, then there
are $p_{n}, q_{n}\in R$ such that $p_{n}Y^{n+1}+q_{n}Z^{n+1}=1$. Indeed, the
proof of this more general statement is exactly the same as that of
Proposition~\ref{prop:4.2}.

This gives rise to the question of whether such a result is also true in a
non-commutative setting. We don't know the answer and formulate the following
open

\begin{ques}
Let $R$ be a non-commutative ring and suppose that $p_{0}, q_{0}, Y, Z$ satisfy
the equation $p_{0}Y+q_{0}Z=1$, where $Y, Z$ are not nilpotent elements. Given a fixed positive integer $n$, do there
exist $p_{n}, q_{n}\in R$ such that $p_{n}Y^{n+1}+q_{n}Z^{n+1}=1$?
\end{ques}

If $Y, Z$ are allowed to be nilpotent elements in $R$, it is easy to show that in general we cannot expect a positive answer. Indeed, let $R=M_{2,2}$ be the ring of $2\times 2$ matrices with integer coefficients and consider the identity $p_{0}Y+q_{0}Z=I$, where
$$
p_{0}=\left(\begin{array}{cc}
                     0 & 0 \\
                     1 & 0 \\
                   \end{array}
                 \right),\;Y=\left(
                   \begin{array}{cc}
                     0 & 1 \\
                     0 & 0 \\
                   \end{array}
                 \right),\; q_{0}=\left(\begin{array}{cc}
                     0 & 1 \\
                     0 & 0 \\
                   \end{array}
                 \right), \;Z=\left(
                   \begin{array}{cc}
                     0 & 0 \\
                     1 & 0 \\
                   \end{array}
                 \right).
$$
However, we have $Y^2=Z^2=0$ and thus, for each $n\geq 1$ and any $p, q\in R$ we have $pY^{n+1}+qZ^{n+1} \neq I$.


\begin{thebibliography}{25}

\bibitem{DP} M.~Davis and H.~Putnam, Diophantine sets over polynomial rings.
{\it Illinois J. Math.} {\bf 7} (1963), 251--256.

\bibitem{De} J.~Denef, The Diophantine problem for polynomial rings and fields
of rational functions. {\it Trans. Amer. Math. Soc.} {\bf 242} (1978),
391--399.

\bibitem{DS1} K.~Dilcher and K.~B.~Stolarsky, Resultants and discriminants of
Chebyshev and related polynomials. {\it Trans. Amer. Math. Soc.} {\bf 357}
(2005), no.~3, 965--981.

\bibitem{DF} D.~S.~Dummit and R.~M.~Foote, {\it Abstract algebra.} Third
edition. Wiley, Hoboken, NJ, 2004.

\bibitem{GI} J.~Gishe and M.~E.~H.~Ismail, Resultants of Chebyshev polynomials,
{\it J. Anal. Appl.} {\bf 27} (2008), no.~4, 499--508.

\bibitem{Go} H.~W.~Gould, {\it Combinatorial Identities\/},
revised edition, Gould Publications, Morgantown, W.Va., 1972.

\bibitem{Go1} H.~W.~Gould, A class of binomial sums and a series transform,
{\it Utilitas Math.} {\bf 45} (1994), 71--83.

\bibitem{KonMoi} J.-M. de Koninck, M. Moineau, Consecutive Integers Divisible by a Power of their Largest Prime Factor  {\it J. Integer Sequence}
Vol. 21 (2018), Article 18.9.3.

\bibitem{Na} J.~Nagura, On the interval containing at least one prime number.
{\it Proc. Japan Acad.} {\bf 28}, (1952), 177--181.

\bibitem{NZM} I.~Niven, H.~S.~Zuckerman, and H.~L.~Montgomery,
{\it An Introduction to the Theory of Numbers\/}, 5th ed., Wiley, 1991.

\bibitem{OEIS}
OEIS Foundation Inc. (2011), {\it The On-Line Encyclopedia of Integer
Sequences},\\
\tt{http://oeis.org}.\rm

\bibitem{PWZ} M.~Petkov{\v s}ek, H.~Wilf, and D.~Zeilberger, {\it A=B}, A K Peters/CRC Press, 1996. Homepage for this book:
{\tt https://www.math.upenn.edu/\~{}wilf/AeqB.html.}

\bibitem{PS} G.~P\'olya and G.~Szeg{\H o}, {\it Problems and theorems in
analysis. I. Series, integral calculus, theory of functions.} Springer-Verlag,
1978.

\bibitem{PZ} M. Pohst and H. Zassenhaus, {\it Algorithmic Algebraic Number
Theory\/}, Cambridge University Press, Cambridge, 1989.

\bibitem{Ri} T.~J.~Rivlin, {\it Chebyshev polynomials. From approximation
theory to algebra and number t–heory.} Second edition. Wiley, New York, 1990.

\bibitem{St} K.~B.~Stolarsky, Discriminants and divisibility for Chebyshev-like
polynomials. {\it Number theory for the millennium, III} (Urbana, IL, 2000),
243--252, A K Peters, Natick, MA, 2002.

\bibitem{Tr1} K.~Tran, Discriminants of Chebyshev-like polynomials and their
generating functions. {\it Proc. Amer. Math. Soc.} {\bf 137} (2009), no.~10,
3259--3269.

\bibitem{Tr2} K.~Tran, Discriminants of polynomials related to Chebyshev
polynomials: the ``Mutt and Jeff" syndrome. {\it J. Math. Anal. Appl.}
{\bf 383} (2011), no.~1, 120--129.

\end{thebibliography}
\end{document}